\documentclass{article} 

    \usepackage[title,titletoc,toc]{appendix}
    \usepackage{amsmath, amssymb, latexsym, amsthm, url}
    \usepackage{comment}
    \usepackage{geometry}
    
      \newtheorem{thm}{Theorem}[section]
      \newcommand{\bthm}{\begin{thm}} \newcommand{\ethm}{\end{thm}}
      \newtheorem{prop}[thm]{Proposition}
      \newcommand{\bprp}{\begin{prop}} \newcommand{\eprp}{\end{prop}}
      \newtheorem{prob}[thm]{Problem}
      \newcommand{\bprb}{\begin{prob}} \newcommand{\eprb}{\end{prob}}
      \newtheorem{lem}[thm]{Lemma}
      \newcommand{\blem}{\begin{lem}} \newcommand{\elem}{\end{lem}}
      \newtheorem{cor}[thm]{Corollary}
      \newcommand{\bcor}{\begin{cor}} \newcommand{\ecor}{\end{cor}}
      \newtheorem{conj}[thm]{Conjecture}
      \newcommand{\bcnj}{\begin{conj}} \newcommand{\ecnj}{\end{conj}}
      \theoremstyle{definition}
      \newtheorem{defn}[thm]{Definition}
      \newcommand{\bdfn}{\begin{defn}} \newcommand{\edfn}{\end{defn}}
      \theoremstyle{remark}
      \newtheorem{rem}[thm]{Remark}
      \newcommand{\brem}{\begin{rem}} \newcommand{\erem}{\end{rem}}
      \newtheorem{cnv}[thm]{Convention}
      \newcommand{\bcnv}{\begin{cnv}} \newcommand{\ecnv}{\end{cnv}}
      \newtheorem{exam}[thm]{Example}
      \newcommand{\bexm}{\begin{exam}} \newcommand{\eexm}{\end{exam}}
      \newcommand{\bpf}{\begin{proof}} \newcommand{\epf}{\end{proof}}
      \newtheorem{exer}[thm]{Exercise}
      \newcommand{\bexer}{\begin{exer}} \newcommand{\eexer}{\end{exer}}
      \newcommand{\ben}{\begin{enumerate}}
      \newcommand{\een}{\end{enumerate}}
      \newcommand{\bit}{\begin{itemize}}
      
      \newcommand{\eit}{\end{itemize}}

    \newcommand{\odeg}[1]{\mathrm{deg}({#1}) }

    \newcommand{\reals}{{\mathbb R}}

    \newcommand{\Z}{{\mathbb Z}}
    \newcommand{\N}{{\mathbb N}}

    \newcommand{\pr}{\mathbb{P}}
    \newcommand{\E}{\mathbb{E}}
    \newcommand{\PP}{\mathbb{P}}
    \newcommand{\om}{\omega}
    \newcommand{\Om}{\Omega}
    \newcommand{\quP}{P}
    \newcommand{\anP}{\mathbb P}
    \newcommand{\ignore}[1]{}
    
    \newcommand{\one}{\mathrm{1}}
    \newcommand{\cG}{\mathcal{G}}
    \newcommand{\cH}{\mathcal{H}}
    \newcommand{\dive}{\mathrm{div}}
    \newcommand{\al}{\alpha}

    \newcommand{\R}{\mathbb{R}}
    \newcommand{\RR}{\mathcal{R}}

\usepackage[T1]{fontenc}
\usepackage[utf8]{inputenc}

\title{Random walks in random hypergeometric environment} 

\author{%
Tal~Orenshtein\footnote{Technische Universit\"at Berlin and Weierstrass Institute, Germany.
{\small orenshtein@wias-berlin.de,}
{\small https://sites.google.com/site/talorenshtein314159/}
}
\and
Christophe~Sabot\footnote{Institut Camille Jordan, Universit\'e Lyon 1, France.
{\small sabot@math.univ-lyon1.fr,}
{\small http://math.univ-lyon1.fr/homes-www/sabot}.
}
}

\date{}

\begin{document}


\maketitle

\begin{abstract}We consider one-dependent random walks on $\Z^d$ in random hypergeometric environment for $d\ge 3$.
    These are memory-one walks in a large class of environments parameterized by positive weights on directed edges and on pairs
    of directed edges which includes the class of Dirichlet environments as a special case.
    We show that the walk is a.s.\ transient for any choice of the parameters, and moreover that the return time has some finite positive moment.
    We then give a characterization for the existence of an invariant measure for the process from the point of
    view of the walker which is absolutely continuous with respect to the initial distribution on the environment in terms
    of a function $\kappa$ of the initial weights.
    These results generalize \cite{sabot2011transience} and \cite{sabot2013particle} on random walks in Dirichlet environment.
    It turns out that $\kappa$ coincides with the corresponding parameter in the Dirichlet case, and so in particular
    the existence of such invariant measures is independent of the weights on pairs of directed edges, and determined solely
    by the weights on directed edges.
\end{abstract}

\section{Introduction}
Despite important progress in the ballistic, balanced, or perturbative regimes (see in particular
\cite{Sznitman-Zerner-99,Sznitman-JEMS-00,Sznitman-PTRF-02,Sznitman-Zeitouni-2006,Bolthausen-Zeitouni-07,Berger-Ramirez-14,Rassoul-Seppalainen-09,Berger-Zeitouni-08,lawler1982weak,guo2012quenched,berger2014quenched}), random walks in i.i.d.\ random environment in dimension $d\ge 2$ remain
a very challenging model. The high non-reversibility of this model is at the heart of the difficulty and several of the basic questions concerning recurrence/transience,
equivalence between directional transience and ballisticity, and diffusive behavior are still unsolved.
The process viewed from the particle, which is a key
tool for reversible models, is still only understood under specific conditions (see \cite{sabot2013particle,Rassoul-Agha-03,Berger-Rosenthal-16}).

The special case of random walks in random Dirichlet environment (RWDE), \cite{Enriquez-Sabot-06}, where the environment is i.i.d.\ at each site
and distributed according to a Dirichlet law, shows remarkable simplifications, while keeping the main phenomenological behavior as the
general model (see \cite{sabot2016survey} for a survey).
For this special choice of distribution, a key property of ``statistical invariance by time reversing'' makes it possible to prove transience
in dimension $d\ge 3$ \cite{sabot2011transience},
existence of an invariant measure viewed from the particle absolutely continuous with respect to the static law,
and equivalence between directional transience and ballisticity in dimension $d\ge 3$ \cite{Sabot-Tournier-11,sabot2013particle,Bouchet-13,sabot2016survey}.

The aim of this paper is to give a generalization of this model and of these results to a class of one-dependent random walks in random environment,
based on some hypergeometric distributions.
The hypergeometric functions defined in \eqref{eq:hypergeometric function} below are a natural special functions constructed from the
Dirichlet distributions.
A generalization of the statistical time-reversal key property is proved (see Corollary \ref{cor:time reversal environment} below),
based on a duality property of these hypergeometric functions.
The latter is a multidimensional generalization of the fact that $\,_2F_1(a,b;c,z)=\,_2F_1(b,a;c,z)$ where $\,_2F_1$ is the basic hypergeometric series
(see e.g.\ \cite{aomoto2011theory}, Section~1.2.1 for the definition and Section~1.3.1 for the integral representation).

This generalization is natural from the following considerations. The statistical time-reversal property mentioned above makes
it possible to write a rather efficient proof of transience and existence of an absolutely continuous invariant measure viewed from the particle in dimension $d\ge 3$, but it fails to give information on some
other natural questions on random walks in random Dirichlet environment (RWDE), such as large deviation and Sznitman's $(T)$ condition.
Nevertheless, in dimension 1 in the Dirichlet case, the large deviation rate function can be
explicitly computed and involves some hypergeometric functions (see \cite{sabot2016survey}, section 8).
The meaning of this computation remains still rather mysterious and the model investigated in this paper comes from
an attempt to generalize the computation done in \cite{sabot2016survey}. Besides, it is also natural to ask to what
extent the strategy used for Dirichlet environments can be generalized. We believe that the class of Dirichlet environments
is the only class of i.i.d.\ environments on which the random walk satisfies the statistical time-reversal property mentioned above.
This paper shows nevertheless that a larger class of environments for \emph{one-dependent} random walks share the same basic features as the Dirichlet environments.








\section{Statement of the results}
\subsection{Hypergeometric functions}\label{subsec:hypergeometric}

Denote by $\Delta^{(n)}:=\left\{u\in(0,1]^n:\displaystyle\sum_{i=1}^n u_i =1\right\}$ the open $n$-simplex.
Define a function on vectors $u\in\Delta^{(n)}$
\begin{equation}\label{eq:hypergeometric density}
\varphi(\alpha,\beta;Z;u) = \left( \displaystyle\prod_{i=1}^n u_i^{\alpha_i - 1} \right) \displaystyle\prod_{j=1}^l  \left( (Z\cdot u)_j ^{-\beta_j} \right)
\end{equation}
where as parameters we take vectors $\alpha\in (\reals_+^*)^n$ and $\beta \in (\reals_+^*)^l$ that satisfy
$\sum_i \alpha_i = \sum_j \beta_j$ and have strictly positive coordinates, and $l\times n$ matrix $Z=(Z_{j,i})$
with strictly positive coefficients, where here and after we use the notation $\reals_+^*=\{t\in\reals:t>0\}$.
Call functions of the following form \emph{hypergeometric functions}:
\begin{equation}\label{eq:hypergeometric function}
\Phi(\alpha,\beta;Z) := \int_{\Delta^{(n)}} \varphi (\alpha,\beta;Z;u) du.
\end{equation}
Here the integral is computed according to the Lebesgue measure on the simplex $du=du_1\cdots du_{n-1}$ so that $u_n=1-\displaystyle\sum_{i=1}^{n-1} u_i$.
When $(Z_{j,i})$ has strictly positive coefficients, we have for all $(Z\cdot u)_j\ge \underline z$, with $\underline z= \min_{i,j}(Z_{j,i})$, so that the integral \eqref{eq:hypergeometric function} is finite.
These functions are classical generalized hypergeometric functions, see e.g. \cite[Section 3.7.4.]{aomoto2011theory}.

\subsection{The model on $\Z^d$}\label{model-Zd}
We denote by $(e_1, \ldots, e_d)$ the canonical base of $\reals^d$, and we set $e_{d+i}=-e_i$ for $i=1, \ldots, d$.
Consider the lattice $\Z^d$ endowed with its natural {\bf directed} graph structure: $\cG_{\Z^d}=(\Z^d, E_{\Z^d})$, where $E=\{(x,x+e_i), \; x\in \Z^d, \; i=1, \ldots, 2d\}$.
The arc graph is the directed graph $\cH_{\Z^d}=(E_{\Z^d}, K_{\Z^d})$, with $K=K_{\Z^d}\subset E_{\Z^d}\times E_{\Z^d}$ given by
$$
K=\big\{\big((x-e_i,x),(x,x+e_j)\big),\; x\in \Z^d, \; i,j=1, \ldots, 2d\big\}.
$$
Concretely, $K$ is the set of couples of succeeding edges that can be crossed by a random walker on the graph $\cG_{\Z^d}$.
The space $\Om_K\subset (0,1]^K$ of random environments on $\cH_{\Z^d}$ is the subspace of transition probabilities of nearest neighbor chains on $\cH_{\Z^d}$ :
$$
\Om_K=\Big\{(\om_{e,e'})_{(e,e')\in K}\in (0,1]^K, \hbox{ such that } \forall e\in E, \; \sum_{e', \; (e,e')\in K} \om_{e,e'}=1\Big\}.
$$
The space $\Om_K$ also naturally describes the space of one-dependent Markov chain kernels on the graph ${\Z^d}$.

Let us now define the random environment. Fix some positive parameters $(\alpha_1, \ldots, \alpha_{2d})$ and a $2d\times 2d$ matrix $Z=(Z_{i,j})$ with strictly positive coefficients. \newline The vectors $(u_{(x,x+e_i)})_{i=1, \ldots, 2d}$, $x\in V$,
are chosen randomly and independently according to the same distribution on the simplex $\Delta^{(2d)}$
with density
\begin{equation}\label{eq:def of environment on Zd}
\frac{1}{\Phi(\alpha,\alpha;Z)} \varphi (\alpha,\alpha; Z ; u) \text{d} u.
\end{equation}
This defines a product law on $(u_{(x,x+e_i)})_{x\in \Z^d, \; i=1, \ldots, 2d}$ which is denote by $\anP^{(\alpha, Z)}$.
Denote by $\E^{(\alpha, Z)}$ the corresponding expectation.
We now define a random environment on $K_{\Z^d}$ by first sampling $(u_{(x,x+e_i)})_{x\in \Z^d, \; i=1, \ldots, 2d}$ according to the last product law and the letting
\begin{equation}\label{eq:def of dpepndent environment on Zd}
\om_{(x-e_i,x),(x,x+e_j)}=\frac{{Z_{i,j} u_{x,x+e_i}}}{\sum_{l=1}^{2d} Z_{i,l} u_{x,x+e_l}}, \,\,\, x\in \Z^d,\, i,j=1,..., 2d.
\end{equation}
Naturally, $\om$ defines the transition probabilities of a Markov chain on the arc graph $\cH_{\Z^d}$, i.e. $w\in \Om_K$, and the distribution $\anP^{(\alpha, Z)}$ induces a probability distribution on the set of environments $\Om_K$.

For an environment $\om$ we denote by $\quP_{e,\om}$ the law of the Markov chain $(X_n)_{n\in \N}$ on state space $E$
started at $e\in E$ with step distribution $\om$.
Whenever $\om$ is sampled according to $\pr^{(\alpha,Z)}$, we say that the last Markov chain is distributed
according to the \emph{quenched law}. Denote by $\anP_{e}^{(\alpha,Z)}$ the marginal of the joint law of the
Markov chain started at $e$ and the environment distributed according to $\pr^{(\alpha,Z)}$. The latter is also called the \emph{averaged law}, or the \emph{annealed law}, of the walk $X$, and it is characterized by
$$\anP_{e}^{(\alpha,Z)} (\cdot) = \int \quP_{e,\om} (\cdot) \text{d} \anP^{(\alpha,Z)}(\om).$$

Remark that from \eqref{eq:def of environment on Zd}, whenever $Z_{i,j}=Z_{i,1}$ for all $i,j= 1, \ldots, 2d$,
then we have $\om_{(x-e_i,x),(x,x+e_j)}=u_{x,x+e_j}$. Therefore, it defines a Markov chain on the original graph
$\cG_{\Z^d}$, and moreover $(u_{x,x+e_i})_{i=1, \ldots, 2d}$ are independent and follow a Dirichlet distribution
with parameters $(\alpha_1, \ldots, \alpha_{2d})$ at each site. Hence, it corresponds to RWDE mentioned in the introduction (for an overview on RWDE see \cite{sabot2016survey}).

\subsection{Order of Green function and Transience on $\Z^d$, $d\ge 3$.}
Fix parameters $(\alpha_i)_{i=1, \ldots, 2d}$ and $(Z_{i,j})_{i,j=1, \ldots, 2d}$ as in Section \ref{model-Zd} and let $\om$ be distributed according to $\anP^{(\alpha,Z)}$.
Denote by $G_\om(e_0,e_0)$ the Green function at $(e_0,e_0)$ of the Markov chain with jump probabilities $\om$,
that is, the $\quP_{e_0,\om}$-expected number of returns to $e_0$.

\begin{thm}\label{thm:green moment} Let $\alpha$ and $Z$ be as in Section \ref{model-Zd} and $d\ge 3$.
Let $\tilde{\kappa}:=\min_{i=1,...,2d}\{\al_{e_i}\}$.
If $s<\tilde{\kappa}$, then $$\E_{e_0}^{(\al,Z)} [G_\om(e_0,e_0)^s]<\infty.$$
In particular, $\om$-$\anP^{(\alpha, Z)}$ almost surely, $(X_n)$ is transient under the quenched law $P_{e_0,\om}$.
\end{thm}

\begin{rem}
A similar statement was proved in \cite[Theorem 1]{sabot2011transience} in the Dirichlet case for $s<\kappa$,
where $\kappa = \max\{2(\sum_{j=1}^{d} \al_{e_j}) - (\alpha_{e_i}-\alpha_{-e_i}):{1\le i \le d}\}$ (an
interpretation of the parameter $\kappa$ is given at the end of Section \ref{sec_stat_invariant}). Hence, the
last theorem generalizes this to the hypergeometric environment in the case $s<\tilde{\kappa}<\kappa$.
The statement would certainly be also true in the case $\tilde{\kappa}\le s <\kappa$: to prove it in this regime, one would need to consider a max-flow type problem adapted to the arc graph $\cH$, as in Section 7.2.\ of \cite{sabot2016survey}
together with our proof of Theorem \ref{thm:point of view}. We don't include that analysis in the current paper, but we stress that it could be done using the same techniques.
\end{rem}
%


\begin{rem} As in the standard Dirichlet case, the case of dimension 2 is still mysterious. It is expected that the walk is recurrent when the weights are symmetric with respect to the axis (i.e. null expected drift at first step), hence the the Green function is a.s. infinite. When the weights are not symmetric, we would expect that there is no long range trapping effect in $d=2$ so that the integrability condition would be the same as in $d\ge 3$. But it is still far from being understood.
In dimension $d=1$, it would be possible to adapt the proof of the Diriclet case (see \cite{sabot2016survey} page 502) to compute the law of the probability starting from the edge $(0,1)$ to never come back to the edge $(0,1)$. It would give that the Green function is integrable for $s<\vert \alpha-\beta\vert$ when $\alpha$ (resp.\ $\beta$) are the weights of the right direction edge (resp.\ left direction edge). The integrability should not depend on the $Z$ parameters. When $\alpha=\beta$ the walk should be recurrent.
\end{rem}

\subsection{Invariant measure for the walker point of view}\label{sec:point of view of the walker}\label{sec_stat_invariant}
Let $(\tau_x)_{x\in\Z^d}$ be the shift maps on $\Om_K$, where $\tau_x(\om(e,e')):=\om(x+e,x+e')$. Here $x+e:=(x+\underline{e}, x+\overline{e}')$ for $x\in\Z^d$ and
$e=(\underline{e},\overline{e})\in E_{\Z^d}$.
  We also let $\tau_e:=\tau_{\overline{e}}$.
 Following the strategy of \cite{kozlov1985method} and \cite{kipnis1986central},
 we define the process
  $$\overline{\om}_n:=\tau_{X_n}(\overline{\om}_0)\, \text{ on } \Om_K$$
  from the point of view of the walker with initial state $\overline{\om}_0 \sim \PP$.
  Under $\PP_{e_0}$, this is a Markov process on $\Om_K$.
Its infinitesimal generator $\RR$ is given by
$$\RR (f)(\om):= \sum_{i=1}^{2d}\om(e_0,e_i) f(\tau_{e_i}(\om)),$$
defined for measurable bounded functions $f$ on $\Om_K$.
 Call a (probability) measure $Q$ on $\Om_K$
 invariant under $\RR$ if $\int \RR f Q(d\om)=\int f Q(d\om)$ for all measurable bounded functions $f$ on $\Om_K$.

The main result of this section is the following generalization of Theorem 1 of
\cite{sabot2013particle}.
\begin{thm}\label{thm:point of view} Let $\kappa := \max\{2(\sum_{j=1}^{d} \al_{e_j}) - (\alpha_{e_i}-\alpha_{-e_i}):{1\le i \le d}\}>0$ and assume $d\ge 3$. Then:
\begin{enumerate}
\item \label{item: kappa larger than 1} If $\kappa > 1$ then there is a unique probability measure $\mathbb{Q}^{(\alpha,Z)}$ on $\Omega_K$ which is invariant under $\mathcal{R}$ and
is absolutely continuous with respect to the initial measure $\anP^{(\alpha,Z)}$.
Moreover, for every $p\in[1,\kappa)$
the Radon-Nikodym derivative $\frac{\mathrm{d}\mathbb{Q}^{(\alpha,Z)}}{\mathrm{d}\anP^{(\alpha,Z)}}$ is in $L_p(\anP^{(\alpha,Z)})$.
(In particular, trivially, the last assertion holds also for every $0<p<1$.)
\item \label{item: kappa smaller than 1} If $\kappa \le 1$ then there is no probability measure satisfying the invariance and absolute continuity properties of the last case.
\end{enumerate}
\end{thm}
The parameter $\kappa$ was considered first in \cite{sabot2011transience} in the context of $\Z^d$,
and was introduced by Tournier \cite{tournier2009integrability} for finite graphs.
Let us give an interpretation of this parameter. If $ S \subset V$ is a nonempty set of vertices, the \emph{outer boundary} of $S$ is defined by
$$\partial_+(S)=\{e\in E: \underline{e}\in S \text{ but } \overline{e}\notin S \}.$$
Define also $\alpha(\partial_+(S)) = \displaystyle\sum_{e\in \partial_+(S) }\alpha_e$,
the total $\alpha$-strength of the edges leaving $S$. Then
\begin{equation}\label{eq:def of kappa}
\kappa = \max\{\alpha(\partial_+(\{0,\overline{e}_i\})):i=1,...,d\}
\end{equation}
represents the maximal weight of the outer boundary of a single edge. Roughly speaking,
it means that the strongest traps in this model are the traps consisting of a single edge, and the strength of these traps is the outer weight.
This last assertion is justified by the following lemma.

\begin{lem}\label{lem:leaving edge infty} Let $T_i=:\inf\{n\ge 0: X_n\notin \{\{0,e_i\},\{e_i,0\}\}\}$, $i=1,...,2d$, be the exist times from the set $\{\{0,e_i\},\{e_i,0\}\}$ of directed edges.
If $\kappa\le 1$, $\E_{e_0}^{(\al,Z)}[T_i]=\infty$ for some $1\le i\le 2d$.
\end{lem}
\begin{proof}
Using \eqref{eq:def of dpepndent environment on Zd} and the independence of the $u_e$ between vertices, and noticing that under $\quP_{e_0,\om}$, $T_i$ is a geometric random variable with expectation $\frac{1}{1-\om(\{0,e_i\}, \{e_i,0\})\om(\{e_i,0\}, \{0,e_i\})}$, the proof is concluded in a similar manner as in \cite[Chapter 3.2]{tournier2009integrability}.
\end{proof}

\begin{rem}\label{lem:leaving edge moments} We believe that the statement of the last lemma can be strengthened to say that $\E_{e_0}^{(\al,Z)}[T_i^s]=\infty$ for some $1\le i\le 2d$ if and only if $s\ge \kappa$. Since the proof should be somewhat involved, and since we shall use only the weak form of the lemma (namely an implication in the case $s=1$), this is not done in the current paper.
\end{rem}

\section{General graphs}\label{sec:Settings}
It is necessary for the proof to define our random environments on general graphs. This is done
in Section \ref{subsec:notations} and \ref{subsec:the model} below.

\subsection{Directed arc graph}\label{subsec:notations}
Remember that a directed graph is connected if for any two vertices $x$ and $y$ there is a directed path connecting $x$ to $y$, or connecting $y$ to $x$.
Let $\cG=(V,E)$ be a connected directed graph with \emph{vertices} and \emph{edges}
such that the in-degrees and out-degrees are finite at each vertex. Here and after in-degree (out-degree) of a vertex $x\in V$
is the number of vertices $y\in V$ that $(y,x)\in E$ (respectively, $(x,y)\in E$).
For each edge $e$ we denote by $\underline e$ and $\overline e$ the tail and head of the edge so that $e=(\underline e, \overline e)$, and we denote by $\check e=(\overline e, \underline e)$ the ``reversed edge''. We denote by $\check \cG=(V, \check E)$ the reversed graph with edge set $\check E:=\{\check e, \; e\in E\}$.

We define the (directed and connected)
\emph{arc graph} $\cH=(E,K)$ with \emph{nodes} $E$ and \emph{arcs} $K$ by setting $K:=\{k=(e,e')\in E^2 : \overline{e}=\underline{e}'\}$.
In words, $\cH$ is the graph so that its nodes are the edges of $\cG$ and its arcs are directed pairs of edges of $\cG$ that share a common vertex, the head of the first edge and the tail of the second one.
Define the reversed graph $\check{\cH}=(\check{E},\check{K})$ by the relation $(\check{e}',\check{e})\in \check{K} \leftrightarrow (e,e')\in K$.
Clearly, $\check \cH$ is also the arc graph of the reversed graph $\check \cG$.

Let $\Om_K\subset (0,1]^K$ be defined by
$$\Om_K = \left\{ \om : \displaystyle\sum_{e':(e,e')\in K}\om(e,e')=1,\, \forall e\in E\right\},$$
seen as a topological (measurable) subspace of $(0,1]^K$ with the standard topology (Borel $\sigma$-algebra).
The space $\Om_K$ will be the space of environments of Markov chains on the directed graph $\cH$. The space
$\Om_{\check{K}}$ is defined similarly for the reversed graph $\check \cH=(\check E,\check{K})$. As in Section \ref{model-Zd},
we note that $\Om_K$ also describes the one-dependent Markov chains on the graph $\cG$.


\subsection{The model on a general directed arc graph} \label{subsec:the model}
Let $\cG=(V,E)$ be a directed connected graph, and let $\cH$ be the corresponding arc graph.
Fix strictly positive parameters $(\alpha_e)_{e\in E}$ and $(Z_{e,e'})_{(e,e')\in K}$. Recall the definition of $\varphi$ and $\Phi$ in Section \ref{subsec:hypergeometric}.
For every $x\in V$, let
\begin{equation}\label{def:phix}
\varphi_x(\al;Z;u)= \varphi((\alpha_{e})_{\underline{e}=x}, (\alpha_e)_{\overline{e}=x};
(Z_{e,e'})_{\overline{e}=x=\underline{e}'} ; (u_e)_{\underline{e}=x})
\end{equation}
be defined for $u$ in the $\odeg{x}$-simplex
\[
\Delta^{(x)}:=\{(u_e)_{\underline{e}=x} : u_e > 0,\,\displaystyle\sum_{\{e:\underline{e}=x\}} u_e =1\}.
\]
Here $\odeg{x}$ is the out-degree of $x$.
Similarly we let, as in \eqref{eq:hypergeometric function},
\begin{equation}\label{def:Phix}
\Phi_x(\alpha;Z) := \int_{\Delta^{(x)}} \varphi_x (\alpha;Z;u)\text{d}_x u=\Phi((\alpha_{e})_{\underline{e}=x}, (\alpha_e)_{\overline{e}=x};
(Z_{e,e'})_{\overline{e}=x=\underline{e}'}) ,
\end{equation}
where $\text{d}_x u =\prod_{\underline{e}=x, e\neq e_x} \text{d} u_e$ is the measure on $\Delta^{(x)}$ defined in Section \ref{subsec:hypergeometric}, where $e_x$ is
an arbitrary choice of edge exiting $x$ (obviously, $du$ does not depend on the choice of $e_x$).
Let $U(x),{x\in V}$, be random vectors with values in $\Delta^{(x)}$, which are independent and
distributed according to the density
\[
\frac{1}{\Phi_x(\alpha,Z)} \varphi_x(\alpha; Z ; u) \text{d}_x u.
\]
For every $e\in E$ let $u_e:=U_e({\underline{e}})$, the $e$ coordinate of the random vector $U({\underline{e}})$. We denote by $\pr^{(\alpha, Z)}$ the distribution on $(u_e)_{e\in E}$ defined in this way.
Denote by $\E^{(\alpha, Z)}$ the corresponding expectation.

From the random variables $u_e$, $e\in E$, we construct an environment $\om\in\Om_K$ by
\begin{equation}\label{eq:def of environment law}
\om(e,e') := \frac {Z_{e,e'} u_{e'}}  {\sum_{e'':\underline{e}''=x} Z_{e,e''} u_{e''} }, \;\;\forall (e,e')\in K.
\end{equation}
With a slight abuse of notation, we also denote by $\pr^{(\alpha,Z)}$ the law thus induced on $\Om_K$. For $\om\in\Om_K$ we denote by $\quP_{e,\om}$
the law of the Markov chain $X$ on $E$ started at $e\in E$ with step distribution $\om$.
Whenever $\om$ is sampled according to $\pr^{(\alpha,Z)}$, the law of the last Markov chain is called the \emph{quenched law}. Denote by $\anP_{e}^{(\alpha,Z)}$ the marginal law
of the joint law of the Markov chain started at $e$ and the environment distributed
according to $\pr^{(\alpha,Z)}$. The latter is also called the \emph{averaged law}, or \emph{annealed law} of the walk $X$, and is characterized by
$$\anP_{e}^{(\alpha,Z)} (\cdot) = \int \quP_{e,\om} (\cdot) \text{d} \anP^{(\alpha,Z)}(\om).$$

Note that, as in the case of $\Z^d$, if $(Z_{e,e'})_{\overline{e}=x=\underline{e}'}$, $x\in V$, are matrices with constant rows (i.e.\ $Z_{e,e'}= c_{e}$ for every $(e,e')\in K$), then $U(x)$
has the $\mathrm{Dirichlet} ( (\alpha_{e})_{\underline{e}=x})$ distribution. Hence $\om$ is an i.i.d\ $\mathrm{Dirichlet} ((\alpha_{e})_{\underline{e}=x})$
environment, and the walk is a standard random walk in Dirichlet environment.

The model defined in Section \ref{model-Zd} on $\Z^d$ obviously corresponds to the case where the parameters $(\alpha_e)_{e\in E}$ and $(Z_{e,e'})_{(e,e')\in K}$ are given by
$$
\alpha_{x,x+e_i}=\alpha_i, \; \forall x\in \Z^{d},\; i=1, \ldots, 2d,  \;\hbox{ and } \; Z_{(x-e_i),(x+e_j)}=Z_{i,j}, \; \forall x\in \Z^{d},\; i,j=1, \ldots, 2d,
$$
with notation as in Section \ref{model-Zd}. We warn the reader about the little confusion of notation between $(\alpha_i)$ and $(\alpha_e)$ and $(Z_{i,j})$ and $(Z_{e,e'})_{(e,e')\in K}$ but we think it will be clear enough from the context. Obviously, the model of Section \ref{model-Zd} describes all the parameters on $\cH_{\Z^d}$ which are invariant by translation, i.e. which satisfy
$
\alpha_{e}=\alpha_{x+e}$, for all $x\in \Z^d$, $e\in E$ and  $Z_{e,e'}=Z_{x+e,x+e'}$,  for all $x\in \Z^d$ and $(e,e')\in K$.

\subsection{A remark on our motivation}
The origin of this work comes from the following fact proved in \cite[Section 8.3]{sabot2016survey}.
In dimension 1 the rate function of the annealed large
deviation principal for the hitting time of a level $k$ is
computed in terms of the hypergeometric function
$\,_2F_1$. The proof is based on the identification of the law of a the solution of a distributional equation,
inspired by Chamayou and Letac, \cite{chamayou_letac_1991}.
The symmetry property of $\,_2F_1$, which is a special case of the duality property proved in Appendix \ref{appendix:Duality of hypergeometric functions},
is at the core of the argument. In the one-dimensional case, this identity generalizes the statistical
time-reversal property. An very interesting problem, which is still open, is to find
a multidimensional counterpart for the rate function formula.

Another motivation is to find other models that share the same type of
statistical time-reversal property with Dirichlet environments. We believe that Dirichlet environments are the only non-trivial model
based on independent transition probabilities at each site that have this property. The model presented
here is a natural extension of the Dirichlet environment that allows one-dependence of the quenched Markov chain
and that shares similar property.


\section{Main tools}\label{sec:tools}


\subsection{Marginal and multiplicative moments}\label{sec:marginal and multiplicative moments}
We assume in this chapter that the graph $\cG$ is finite.
Our first observation regarding the hypergeometric distribution is the distribution of its marginal.
A direct computation gives that if $\om$ is defined as in \eqref{eq:def of environment law}, then we have for $e,e'$ so that $\overline{e}=x=\underline{e}'$
\begin{equation}\label{eq:moment of om - Beta-like}
\E^{(\alpha,Z)}[\om(e,e')^s] = Z_{e,e'}^s  \frac{ \Phi_x (\alpha+s(\delta_{e}+\delta_{e'}),Z)}{\Phi_x(\alpha,Z)}.
\end{equation}
In particular we see that the above is finite whenever the arguments of $\Phi_x$ is strictly positive,
and in particular as long as $s>-\min\{\alpha_{e},\alpha_{e'}\}$.
Note that in the Dirichlet case, e.g.\ whenever $Z\equiv 1$, we have that $\om(e,e')=u_{e'}$ has the
Beta distribution $\mathrm{Beta}(\alpha_{e'}, \sum_{\overline e = x}\alpha_e - \alpha_{e'})$.

Next, we shall expand the definition of the measure $\PP^{(\alpha,Z)}$ on environments to include
a possibility to increase or decrease the weights
$\alpha$ and $Z$.

Assume here that $\cG$ is finite.
For a function $\xi:K\to \reals$ let
$$\overline{\xi}_e:=\displaystyle\sum_{\underline{e}'=\overline{e}}\xi(e,e')\,\, \text{  and  }\,\,
\underline{\xi}_{e'}:=\displaystyle\sum_{\overline{e}=\underline{e}'}\xi(e,e')$$
be the total `weight' leaving $e$, and entering $e'$, respectively.

We now define the measure  $\PP^{(\alpha,\xi,Z)}$ on $\Om_K$ by a similar procedure.
For every $x\in V$ and  $u\in \Delta^{(x)}$ we let
\[
\varphi_x(\al;\xi;Z;u)= \varphi(
(\alpha_{e}+ \underline{\xi}_{e})_{\underline{e}=x},
(\alpha_e + \overline{\xi}_e )_{\overline{e}=x};
(Z_{e,e'})_{\overline{e}=x=\underline{e}'} ;
(u_e)_{\underline{e}=x}),
\]
and similarly
\[
\Phi_x(\alpha;\xi;Z) := \int_{\Delta^{(x)}} \varphi_x (\alpha;\xi;Z;u)\text{d}_x u.
\]
This is well-defined as long as $\alpha_e + \overline{\xi}_e>0$ and $\alpha_e + \underline{\xi}_e>0$ for all $e\in E$.
Next, $U_x,{x\in V}$, are taken to be independent with density
\[
\frac{1}{\Phi_x(\alpha;\xi;Z)} \varphi_x(\alpha;\xi;Z ; u) \text{d}_x u.
\]
Putting $u_e:=U_{\underline{e}}(e)$, $e\in E$, and constructing $\om\in\Om_K$ as in \eqref{eq:def of environment law},
we denote its quenched and annealed laws by $\quP_{e_0,\om}$ and $\PP^{(\alpha,\xi,Z)}_{e_0}$.
Note that in the case $\xi\equiv 0$ we have $\PP^{(\alpha,0,Z)}_{e_0}=\PP^{(\alpha,Z)}_{e_0}$.

It will be beneficial to define
\begin{equation}\label{eq: def of product of hypergeometric on all vertices}
 F(\alpha;\xi;Z):=\prod_{x\in V}\Phi_x(\alpha;\xi;Z),\,\text{ and } F(\alpha;Z):=F(\alpha;0;Z).
\end{equation}
Also, for functions $\beta,\gamma:A\to\reals_+$ so that $A$ is a finite set and $\beta$ is strictly positive, we define
\begin{equation}\label{eq:def of power of functions}
\beta^{\gamma} := \displaystyle\prod_{x\in A} \beta(x)^{\gamma(x)}.
\end{equation}
A direct computation gives that for every $\xi,\Theta:K\to\reals$
\begin{equation}\label{eq:expectation of om to the xi}
\E^{(\alpha,\Theta,Z)}[\om^\xi] = Z^{\Theta+\xi} \cdot \frac{ F (\alpha;\Theta+\xi;Z)}{F(\alpha;\Theta;Z)},
\end{equation}
as long as the right hand side of the equation is well defined.

If we think of $\PP^{(\alpha,Z)}$ as the law of $(u_e)_{e\in E}$, i.e. a measure on $\prod_{x\in V} \Delta ^{(x)}$, then the Radon-Nikodym derivative one gets by changing the values of $\alpha$ is explicit. Indeed, for $\theta: E\to\reals_+$ so that $\al_e>\theta_e$ for all $e\in E$, and for any random variable $Y(\om)=(Y\circ\om)(u)$
\begin{equation}\label{eq:radon-Nikodym om plus theta}
\E^{(\alpha,Z)}[Y] = \frac{ F (\alpha+\theta,Z)}{F(\alpha,Z)} \E^{(\alpha + \theta,Z)}[ \tilde{u}^{-\theta}\cdot Y],
\end{equation}
where $$\tilde{u}_e:=\frac{u_e}{\sum_{\underline{e'}=\overline{e}} Z_{e,e'}u_{e'}}.$$

\subsection{Duality formula}\label{subsec:duality formula}

A key feature of the hypergeometric functions defined in \eqref{eq:hypergeometric function} is the following duality
formula \cite[Page 169]{aomoto2011theory}, which has consequences regarding time-reversing.
This will be discussed in Chapter \ref{sec:time reversal}, and a direct proof of Lemma \ref{lem:duality}
will be supplied in Appendix \ref{appendix:Duality of hypergeometric functions}.
Define
\begin{equation}
B(\alpha)=B(\alpha_1,...,\alpha_n)=\frac{\prod_{i=1}^{n}\Gamma(\alpha_i)}{\Gamma\left(\sum_{i=1}^{n}\alpha_i\right)},
\end{equation}
where $\Gamma$ is the standard Gamma function, i.e.\ $\Gamma(t)=\int_{0}^{\infty}x^{t-1}e^{-x}dx$.
\begin{lem}[Duality formula]\label{lem:duality} With the notation from \eqref{eq:hypergeometric function}, the following holds as soon as $\sum_{i=1}^{n} \alpha_i=\sum_{j=1}^l \beta_j$
\[
B(\alpha)^{-1} \Phi(\alpha,\beta,Z) = B(\beta)^{-1} \Phi(\beta, \alpha,Z^t),
\]
where $Z^t$ is the transposed matrix corresponds to $Z$.
\end{lem}
We remark that in the Dirichlet case (e.g., whenever $Z\equiv 1$) both $\Phi(\alpha,\beta,Z)=B(\alpha)$ and $\Phi(\beta, \alpha,Z^t)=B(\beta)$
and so in this case the duality is trivial.


\subsection{Time-reversal statistical invariance}\label{sec:time reversal}

In this section we assume that the graph $\cG=(V,E)$ is finite. For $\om\in\Om_K$, let $\pi^\om=(\pi^\om(e)))_{e\in E}$ be the invariant probability measure of the Markov chain on $E$ with transition probabilities $\om$.
(Note that by ellipticity of $\om$, the finite state Markov chain is a.s.\ irreducible and hence $\pi^{\om}$ is a.s.\ unique.)
Define the time reversed environment $\check{\om}\in \Omega_{\check{K}}$ by letting
\begin{equation}\label{eq:def of om check}
\check{\om}(\check{e}',\check{e}) = \pi^{\om}(e) \om(e,e') \pi^{\om}(e')^{-1}.
\end{equation}
Let $\check{\pi}^{\check{\om}}$ be the invariant probability measure of the Markov chain on $\check{E}$ with
transition probabilities $\check{\om}$. Then, since $\check{\pi}^{\check{\om}}$
is also the invariant probability measure of the time reversed chain defined by $\omega$, we have
\begin{equation}\label{eq: time reversed invariant measure}
\check{\pi}^{\check{\om}}(\check{e})=\pi^{\om}(e)
\end{equation}
for every $e\in E$.
Note that $\check{\om}$ is an element of $\Omega_{\check K}$.

Let $\check{\alpha}_{\check{e}}:=\alpha_e$ for every $e\in E$. Also, denote $\check{Z}$ the `reversed' matrices corresponds to $Z$, that is $\check{Z}_{\check{e'},\check{e}} = (Z^t)_{e',e} = Z_{e,e'}$.
Let $C=\{e_0,e_1,...,e_n=e_0\}$ be a cycle in $\cH$, $n=n(C)$ is its length. (The reader should notice that here $C$ is a cycle of edges,
and so viewed as a sequence of vertices it has the form
$\{ \underline{e_0}, \underline{e_1},\underline{e_2},...,\underline{e_n}=\underline{e_0}, \overline{e_0}=\underline{e_1}\}$,
i.e., a cycle of vertices plus a repetition of the vertex $\underline{e_1}$.)
Define $\check{C}:=\{\check{e}_n,\check{e}_{n-1},...,\check{e}_0=\check{e}_n\}$ to be the corresponding reversed
cycle in $\check{\cH}$.
For a finite collection $\mathcal{C}$ of cycles we denote by $\check{\mathcal{C}}:=\{\check{C}: C\in \mathcal{C}\}$.
Set $\om_{C} := \prod_{k=0}^{n-1} \om_{e_k,e_{k+1}}$, and $\om_{\mathcal{C}}:=\prod_{C\in\mathcal{C}}\om_C$.
By \eqref{eq:def of om check}, we have
$$
\om_C=\check\om_{\check C},
$$
for all cycles $C$.
Similarly, we set $Z_{C} := \prod_{k=0}^{n-1} Z_{e_k,e_{k+1}}$ and $Z_{\mathcal{C}}:=\prod_{C\in\mathcal{C}}Z_C$.
We have, by definition of $\check Z$, that $Z_{C}=\check Z_{\check C}$ for all cycle $C$.

We introduce now the divergence operator on the graph $\cG$: we define $\dive :\R^E\mapsto \R^V$ by
$$
\dive(\theta)(x)=\sum_{\underline e=x} \theta(e)-\sum_{\overline e=x} \theta(e), \;\;\; \forall \theta\in \R^E.
$$

\begin{lem}\label{lem:cycles coincide} Assume $\dive(\alpha)=0$. The following hold for all finite collections of cycles ${\mathcal C}$,
\[
\E^{(\alpha,Z)} (\check{\om}_{\check{\mathcal{C}}})=\E^{(\check{\alpha},\check{Z})}(\om_{\check{\mathcal{C}}}).
\]
\end{lem}
\begin{proof}
Denote by $N_e=N_e(\mathcal{C})$ the number of $0\le k\le n-1$, so that $e=e_k$, where $e_k\in C$, for some $C\in \mathcal{C}$ of length $n=n(C)$.
We denote similarly $\check N=\check N(\mathcal{C})$ the corresponding counting function for the collection of reversed cycles. Clearly, $N_e=\check N_{\check e}$.

A direct computation gives
\begin{eqnarray}\label{eq:prob-cycle}
\E^{(\alpha,Z)}(\om_\mathcal{C}) &=& Z_{{\mathcal C}}
\displaystyle\prod_{x\in V} \frac { \Phi_x (\alpha + N,Z) } { \Phi_x(\alpha,Z) }= Z_{{\mathcal C}}
\frac {F(\alpha + N,Z) } {F(\alpha,Z) }.
\end{eqnarray}
Indeed, from the definition of the environment $\om$, see \eqref{eq:def of environment law}, we have
$$
\om_\mathcal{C} = Z_{{\mathcal C}} \prod_{x\in V}\left(\prod_{e', \, \underline{e}'=x}  {U_{e'}^{N_{e'}}}\right)\left(\prod_{e, \, \overline{e}=x}\left( \sum_{\underline{e}''=x} Z_{e,e''}U_{e''}\right)^{-N_{e}}\right),
$$
the term $Z_{{\mathcal C}}$ coming from the term $Z_{e',e}$ in  \eqref{eq:def of environment law}, the second term coming from the times when the cycle enters $e'$, the last term coming from the times when the cycle leaves $e$.
Combined, with the definitions  \eqref{eq:hypergeometric density}, \eqref{def:phix}, \eqref{def:Phix}, \eqref{eq: def of product of hypergeometric on all vertices}, it gives \eqref{eq:prob-cycle}.

Next, since $\dive(\alpha)=0$, the Duality formula Lemma \ref{lem:duality} says that for all $x\in V$,
\begin{small}
\begin{eqnarray*}
\Phi_x(\alpha,Z) = \Phi((\alpha_e)_{\underline e=x},(\alpha_e)_{\overline e=x},Z)=
\frac{B((\alpha_e)_{\underline e=x})}{B((\alpha_e)_{\overline e=x})} \Phi((\alpha_e)_{\overline e=x}, (\alpha_e)_{\underline e=x},Z^t)=
\frac{B((\alpha_e)_{\underline e=x})}{B((\alpha_e)_{\overline e=x})}
\Phi_x(\check \alpha,\check Z).
\end{eqnarray*}
\end{small}
It implies that,
$$
F(\alpha, Z)= {\frac{G(\alpha)}{G(\check \alpha)}} F(\check \alpha, \check Z).
$$
where,
$$
G(\alpha):= \prod_{x\in V} B((\alpha_e)_{\underline e=x}).
$$
Since $\dive(\alpha)=0$, we have for all $x\in V$, $\sum_{\underline {e} = x}\alpha_e=\sum_{\overline {e} = x}\alpha_e$.
Therefore,
\begin{footnotesize}
\[
G(\alpha)
=\displaystyle\prod_{x\in V}\frac {\prod_{\underline {e} = x}\Gamma(\alpha_e) } {\Gamma\left(\sum_{\underline {e} = x}\alpha_e\right) }\\
=\frac {\prod_{e\in E} \Gamma(\alpha_e) }
{\prod_{x\in V}\Gamma \left(\sum_{\underline {e} = x}\alpha_e\right) }
\\
=\frac {\prod_{e\in E} \Gamma(\alpha_e) }
{\prod_{x\in V}\Gamma\left(\sum_{\bar {e} = x }\alpha_e\right)}
=\displaystyle\prod_{x\in V}\frac {\prod_{\overline {e} = x}\Gamma(\alpha_e) } {\Gamma\left(\sum_{\overline {e} = x}\alpha_e\right) }
=G(\check\alpha),
\]
\end{footnotesize}
where in the last equality we used the fact that $\check\alpha_{\check{e}} = \alpha_e$.
Hence, $F(\alpha, Z)=F(\check \alpha, \check Z)$.

Since ${\mathcal{C}}$ is a collection of cycles, it implies that $\dive(N)=0$, the same applies for $\alpha+N$ and we get $F(\alpha+N, Z)=F(\check \alpha+\check N, \check Z)$. From \eqref{eq:prob-cycle} and since $Z_{\mathcal C}=\check Z_{\check {\mathcal C}}$, we deduce
$$
\E^{(\alpha,Z)}(\check\om_{\check{\mathcal{C}}})=\E^{(\alpha,Z)}(\om_\mathcal{C})= Z_{{\mathcal C}}
\frac {F(\alpha + N,Z) } {F(\alpha,Z) }= \check Z_{{\check {\mathcal C}}}
\frac {F(\check \alpha + \check N,\check Z) } {F(\check \alpha,\check Z) }=
\E^{(\check \alpha,\check Z)}(\om_{\check {\mathcal{C}}}).
$$

%
 %
\end{proof}

%

\begin{cor}\label{cor:time reversal environment}
Let $\om \sim \pr^{(\alpha,Z)}$. The time-reversing function $\om\mapsto \check{\om}$, where $\check{\om}$ is defined as in \eqref{eq:def of om check}, defines a new law $P$ on $\Om_{\check{K}}$.
Then, if $\dive(\alpha)=0$,
\[
 P=\pr^{(\check{\alpha},\check{Z})}
 \]
\end{cor}
\begin{proof}
Since $\om(e,e')$ and $\pi^\om(e)$ are positive and bounded by $1$, and $E$ and $K$ are finite, the law of $P$ is determined by its moments.
That is, it's enough (and actually equivalent) to show that for any $\eta:\check{K}\to\Z_+$
$\E^{({\alpha},{Z})} [\check{\om}^\eta ] = \E^{(\check{\alpha},\check{Z})} [\om^{{\eta}}]$.
Note that since the graph is finite and all $\om(e,e')\in(0,1)$, under the quenched law the Markov chain and its time reversal are both recurrent.
But now notice that the law of the recurrent Markov chain $\check \om$ is determined by the law of its cycles.
Indeed,
for all $(e,e')\in K$,
$\check \om(\check e',\check e)=\sum_{C\in \mathcal{C}_{e,e'}}\check \om_{\check C}$,
where ${\mathcal{C}_{e,e'}}$ is the family of all cycles $C$ starting at $e$, going immediately to $e'$ and returning to $e$ for the first time.
It clearly implies that if $\eta=(\eta_{e,e'})_{(e,e')\in \check K}$ is a positive vector, then $\check \om^\eta$
can be written as a sum
with positive coefficients of terms of the type $\check \om_{\check{\mathcal C}}$, where $\mathcal C$ are finite collections of cycles.
Using Lemma \ref{lem:cycles coincide}, it implies that
$$
\E^{(\alpha, Z)}\left(\check \om^\eta\right)= \E^{(\check \alpha, \check Z)}\left(\om^\eta\right).
$$

\end{proof}

We finish with an application from the proof of the last corollary.
Set $H_{e}:=\inf\{n\ge 0, \;\; X_n=e\}$ and $H_{e_0}^+:=\inf\{n>0, \;\; X_n=e_0\}$.
\begin{cor}\label{cor:application of time reversal}
For every $(e,e_0)\in K$, for $\om \in K$,
$$
\quP_{e_0,\om}[X_{H_{e_0}^+-1}=e]
 =
\quP_{\check{e}_0,\check{\om}}[X_1=\check{e}].
$$
\end{cor}
\begin{proof}
As in the last corollary, it follows from the fact that the weights are strictly positive $\pr^{(\alpha,Z)}$-a.s., that the Markov chains on the finite graphs $\cH, \check{\cH}$ are recurrent.
Hence the probability $\quP_{\check{e}_0,\check{\om}}[X_1=\check{e}]$ equals to the sum of the $\check{\om}$ weight over of all cycles $\{\tilde e_1,...,\tilde e_n\}$ with $\tilde e _1=\tilde e_n=\check{e}_0$ but $\tilde e_i \ne \check{e}_0$ for $ 1<i<n$, and $\tilde e_2 = \check{e}$. To end one notices that the sum of $\om$ weight over the reversed cycles
gives exactly $\quP_{e_0,\om}[X_{H_{e_0}^+-1}=e]$, and by Lemma \ref{lem:cycles coincide} these probabilities are equal.
\end{proof}

\subsection{Arc graph identities }\label{sec:arc graph identities}

We now use the same notation for the \emph{divergence} operator on $\cG$ also for the arc graph $\cH$.
$\mathrm{div}:\reals^K \to \reals^E$ is defined by
\begin{equation}\label{eq: def of divergence}
\mathrm{div} (\Theta) (e) = \displaystyle\sum_{e':(e,e')\in K} \Theta(e,e') - \displaystyle\sum_{e':(e',e)\in K} \Theta(e',e),
\end{equation}
for $\Theta:K\to\reals$ and $e\in E$.
We also denote by $\check{\Theta}:\check{K}\to\reals_{+}$ the function so that $\check{\Theta}((\check{e}',\check{e}))=\Theta((e,e'))$.
With a minor abuse of notation the divergent is analogous defined as $\mathrm{div}:\reals^{\check{K}} \to \reals^ {\check{E}}$. This gives
\begin{equation}\label{eq: div of theta and checktheta}
\mathrm{div} (\Theta) (e) = - \mathrm{div} (\check{\Theta}) (\check{e})
\end{equation}
for every $\Theta:K\to\reals_{+}$ and $e\in E$.
\begin{lem}\label{lem:identity for theta to the div over the same with the checks}
The following formula holds for every $\om\in\Om_K$ and $\Theta:K\to\reals_{+}$:
$$\frac{\check{\om}^{\check{\Theta}}} {\om^\Theta} = (\pi^{\om})^{\mathrm{div} \Theta}.$$
\end{lem}

\begin{proof}
Indeed,
\begin{eqnarray*}
\frac{\check{\om}^{\check{\Theta}}} {\om^\Theta} &=&
\frac{ \displaystyle\prod_{(\check{e}',\check{e})\in \check{K}}  \check{\om}(\check{e}',\check{e})^{\check{\Theta}(\check{e}',\check{e})}} { \displaystyle\prod_{(e,e')\in K}  \om(e,e')^ {\Theta(e,e')}}   \\
&=&     \displaystyle\prod_{(e,e')\in K} \frac{ \check{\om}(\check{e}',\check{e})^{\check{\Theta}(\check{e}',\check{e})}} { \om(e,e')^ {\Theta(e,e')}}  \\
&=&     \displaystyle\prod_{(e,e')\in K} \frac{ (\pi^{\om}(e) \om(e,e') \pi^{\om}(e')^{-1})^{\Theta(e,e')}}  { \om(e,e')^ {\Theta(e,e')}} \\
&=&     \displaystyle\prod_{e\in E} \pi^{\om}(e)^{ \left(\sum_{e':(e,e')\in K} \Theta(e,e') - \sum_{e':(e',e)\in K} \Theta(e',e)\right) }   \\
&=& (\pi^{\om})^{\mathrm{div} \Theta}.
\end{eqnarray*}
\end{proof}

\subsection{Flows}\label{subsec:flows}

\subsubsection*{Flow identity}\label{subsec:flows identities}

For $e_0, e\in E$ and $\gamma>0$, a \emph{flow from $e_0$ to $e$} of strength $\gamma$, is a function $\Theta:K\to\reals_{+}$ such that
$$\mathrm{div}(\Theta) = \gamma (\delta_{e_0} - \delta_{e}).$$

$\Theta:K\to\reals_{+}$ is a \emph{total flow from $e_0$} of strength $\gamma$ if it has the form
$$\mathrm{div}(\Theta) = \gamma \displaystyle\sum_{e\in E}(\delta_{e_0} - \delta_{e}).$$

\begin{lem}\label{lem:identity for pi to the div if theta is a total flow}
If $\Theta:K\to\reals_{+}$ is a total flow from $e_0$ of strength $\gamma$, then
$$ (\pi^{\om})^{\mathrm{div}(\Theta)} = \displaystyle\prod_{e\in E}\left ( \frac{\pi^{\om}(e_0)}{\pi^{\om}(e)}\right)^\gamma.$$
\end{lem}

\begin{proof}
First note that  \begin{equation*}
 \gamma \displaystyle\sum_{e'\in E}(\delta_{e_0} - \delta_{e'})(e)= \gamma
    \begin{cases}
    (|E|-1) & \text{if } e = e_0 \\
    -1 & \text{ if } e\neq e_0.
    \end{cases}
\end{equation*}
Hence,
  \begin{eqnarray*}
(\pi^{\om})^{\mathrm{div}(\Theta)} &=&  \displaystyle\prod_{e\in E} \pi^{\om}(e)^{\mathrm{div}(\Theta)(e)} \\
&=& \displaystyle\prod_{e\in E} \pi^{\om}(e)^{\gamma \sum_{e'\in E}(\delta_{e_0} - \delta_{e'})(e)}\\
&=& \pi^{\om}(e_0)^{\gamma|E|} \displaystyle\prod_{e\in E} \pi^{\om}(e)^{-\gamma}       \\
&=& \displaystyle\prod_{e\in E}\left(\frac{\pi^{\om}(e_0)}{\pi^{\om}(e)}\right)^\gamma.
\end{eqnarray*}

\end{proof}

\subsubsection*{Construction of good flows}\label{subsec:flow construction}
Consider first the lattice $\cG_{\Z^d}=(\Z^d, E_{\Z^d})$ (see Section \ref{model-Zd}). Let $(c(e))_{e\in E}$ be a
set of positive weights on the edges, called capacities.
A finite subset $S\subset E$ is called a \emph{cutset separating 0 from infinity} if any infinite
simple directed path starting at 0 crosses at least one directed edge of $S$ (simple means that the path never
visits the same vertex twice).
The \emph{mincut} of the graph $\cG_{\Z^d}$ with capacities $(c(e))$ is the value
$$
m(c)=\inf\left\{\sum_{e\in S} c(e) : S \text{ is a cutset}\right\}.
$$

Let $T_N=(V_N,E_N)$ be the $N$-torus graph in $d$ dimensions, that is the associated directed graph
image of $\Z^d$ by projection on $(\Z/N\Z)^d$.
We identify the edge set $E_N$ with the edges $e$ of $E_{\Z^d}$ such that
$\underline e\in[-N/2, N/2)^d$. Let $\mathcal H =\mathcal H_N=(E_N,K_N)$ be the corresponding arc graph.
The following lemma supplies a total flow on the arc graph with good properties, and is a consequence of
the max-flow min-cut theorem together with the transience of $\Z^d$, $d\ge 3$.
\begin{lem}[Min-cut total flow on $\cH$]\label{lem: Construction of total flow of edges}
Let $d\ge 3$. Assume that $(c(e))_{e\in E_{\Z^d}}$ is uniformly bounded, i.e. there exist some constants $0<C_1<C_2<\infty$ such that $C_1\le c(e)\le C_2$ for all edge $e$. Fix $e_0$ to be an edge with $\overline{e}_0=0$. There is a constant $c_2$ so that for every large enough $N$ there is a non-negative function $\Theta=\Theta_N$ on $K_N$
with the following properties:
\begin{enumerate}
 \item $\overline{\Theta}_e \le c(e)+  m(c) \one_{e=e_0}$ (almost below the capacity).
 \item $\displaystyle\sum_{(e,e')\in K_N}\Theta(e,e')^2 < c_2$ (bounded $L_2$ norm).
 \item $\mathrm{div}(\Theta)=\frac{m(c)}{dN^{d}}\displaystyle\sum_{\tilde{e}\in E_N}(\delta_{e_0}-\delta_{\tilde{e}})$ (total flow from $e_0$).
 \end{enumerate}
 where $m(c)$ is the \textrm{min cut} of the network $c$.
 \end{lem}

For the proof we shall use the analogous
\begin{lem}\cite[Lemma 2]{sabot2013particle}\label{lem:sabot flow on vertices}
Let $d\ge 3$.  Assume that $(c(e))_{e\in E_{\Z^d}}$ is uniformly bounded, i.e. there exist some constants $0<C_1<C_2<\infty$ such that $C_1\le c(e)\le C_2$ for all edge $e$.
There is a constant $c_1$ so that for every large enough $N$ there is a non-negative function $\theta=\theta_N$ on $E_N$
with the following properties:
\begin{enumerate}
 \item $\theta(e)\le c(e)$ (below the capacity).
 \item $\displaystyle\sum_{e\in E_N}\theta(e)^2 < c_1$ (bounded $L_2$ norm).
 \item $\mathrm{div}(\theta)(y)=\overline{\theta}_y-\underline{\theta}_y=\frac{m(c)}{N^{d}}\displaystyle\sum_{x\in V_N}(\delta_{0}-\delta_{x})(y)$ (total flow from $0$).
 \end{enumerate}
 where $m(c)$ is the \textrm{min cut} of the network $c$.
\end{lem}

\begin{proof}[Proof of Lemma \ref{lem: Construction of total flow of edges}]
Fix $N\ge 2$ and let $\theta$ be according to Lemma \ref{lem:sabot flow on vertices}. Write simply $m=m(c)$. We define $\Theta=\Theta_N: K_n \to\reals_+$
by
\begin{equation}\label{eq:def of good Theta from a good theta}
 \Theta(e,e')=\frac{(\theta(e)+m \one_{e=e_0}) (\theta(e') + \frac{m}{dN^d}) } {\overline{\theta}_{\overline e} + \frac{m}{N^d}},\;\;\; (e,e')\in K_N.
\end{equation}
We claim that $\Theta$ satisfies the assertions of the lemma.
First note that by property 1 of Lemma \ref{lem:sabot flow on vertices} $\overline\Theta_e = \theta(e)+m \one_{e=e_0} \le c(e) + m \one_{e=e_0}$.
Next, by \eqref{eq:def of good Theta from a good theta}, $\Theta(e,e')\le \theta(e)+m\one_{e=e_0}$.
Therefore, by property 2 of Lemma \ref{lem:sabot flow on vertices}
$$\displaystyle\sum_{(e,e')\in K_N}\Theta(e,e')^2 \le \displaystyle\sum_{e\in E_N} 2d (\theta(e) +m\one_{e=e_0})^2 < 2dc_1+ (c(e_0)+m)^2)=:c_2.
$$
To end, by property 3) in Lemma \ref{lem:sabot flow on vertices} we have
\begin{eqnarray*}
\underline\Theta_{e'}
&=& (\theta(e') + \frac{m}{dN^d}) \displaystyle\sum_{e:e\to e'} \frac{\theta(e)+m \one_{e=e_0}}{\underline{\theta}_{\underline{\theta'}}+\frac{m}{dN^d}} \\
&=& (\theta(e') + \frac{m}{dN^d}) \frac{\underline{\theta}_{\underline{e'}} + m \one_{\underline{e'}=0}} {\overline{\theta}_{\underline{e'}} + \frac{m}{dN^d} } \\
&=& \theta(e') + \frac{m}{dN^d}.
\end{eqnarray*}
Hence,  $\mathrm{div}(\Theta)= m \one_{e=e_0} - \frac{m}{dN^d} = \frac{m}{dN^{d}}\displaystyle\sum_{\tilde{e}\in E_N}(\delta_{e_0}-\delta_{\tilde{e}})$.

\end{proof}

\section{The Green function has a positive moment}\label{sec:proof of transience}
In this section we prove Theorem \ref{thm:green moment}. The proof follows closely the ones in
\cite{sabot2011transience} and in \cite[Section 7.2.]{sabot2016survey}.
\noindent Fix $e_0$ of the form $e_0=(x_0,0)$. Let $N\in\N$ and define $G_N$ to be the graph with vertices $V_N=B(0,N)\cup\{\partial\}$, where $\partial$ is an additional vertex and $B(0,N)$ denotes a box  in $\Z^d$ with side length $N$ around the origin, and edges $\hat E_N=E_N\cup\{(\partial, x_0)\}$, i.e.~of the following types.
The edges set $E_N$ is the set of
directed edges between neighboring vertices inside $B(0,N)$ (as in $\Z^d$) and between the vertices in the inner boundary of $B(0,N)$ and $\partial$. (I.e., we identify all vertices on the boundary of $B(0,N)$ with the special vertex $\partial$.)
We also add to $E_N$ one special edge $(\partial,x_0)$. Denote by $\cH_N=(\hat E_N,K_N)$ the corresponding arc graph.

\noindent The weights  $\alpha$ and $Z$ on $\Z^d$ naturally yield weights on $E_N$. We endow the special edge $(\partial, x_0)$ with weight $\alpha_{(\partial, x_0)}=\gamma$, for some $\gamma>0$ that will be defined later on, and set $Z_{e,e'}=1$ whenever $\overline{e}=\partial$. Set also $Z_{((\partial, x_0),e_0)}=1$.
With this choice we note that on $\hat E_N$
$$
\mathrm{div} (\alpha)=\gamma(\delta_\partial-\delta_0).
$$

\noindent Consider now a unit flow $\theta:E_N\to \R_+$ from 0 to $\partial$ (i.e.\ $\mathrm{div}(\theta)(e)=\delta_0-\delta_\partial$) and assume that $0\le \theta\le 1$.
Extend $\theta$ by 0 on the special edge $(\partial, x_0)$.
We consider $\alpha + \gamma \theta$. These weights give a flow with null divergence on $\hat E_N$.

\noindent Set $H_{(\partial,x_0)}:=\inf\{n\ge 0, \;\; X_n=(\partial,x_0)\}$ and $H_{e_0}^+:=\inf\{n>0, \;\; X_n=e_0\}$.
We can now apply Corollary~\ref{cor:application of time reversal} on $G_N$ to get that under the law $\pr_{}^{(\alpha+\gamma \theta,Z)}$
\begin{eqnarray*}
\quP_{e_0,\om}[H_{(\partial,x_0)}< H_{e_0}^+]
& \ge &
\quP_{e_0,\om}[X_{H_{e_0}^+-1}=(\partial,x_0)]\\
& = &
\quP_{\check{e}_0,\check{\om}}[X_1=(x_0,\partial)].
\end{eqnarray*}

Hence, using \eqref{eq:moment of om - Beta-like}, we have for $\epsilon>0$
\begin{eqnarray*}
\E^{(\alpha+\gamma \theta,Z)} [(\quP_{e_0,\om}[H_{(\partial,x_0)}< H_{e_0}^+])^{-\epsilon}]
&\le&
\E^{(\alpha+\gamma \theta ,Z)}[ (\check{\om} (\check{e}_0,(x_0,\partial)))^{-\epsilon}]\\
&=&
\E^{(\check\alpha+\gamma \check\theta ,\check Z)} [ (\om (\check{e}_0,(x_0,\partial)))^{-\epsilon}]\\
&=&
\check{Z}_{(\check{e}_0,(x_0,\partial))}^{-\epsilon}  \frac{ \Phi_{x_0} (\check\alpha+\gamma \check\theta - \epsilon(\delta_{\check{e}_0}+\delta_{(x_0,\partial)}),\check{Z})}{\Phi_{x_0}(\check{\alpha},\check{Z})}\\
&=&
\frac{ \Phi_{x_0} (\check\alpha+\gamma \check\theta - \epsilon(\delta_{\check{e}_0}+\delta_{(x_0,\partial)}),\check{Z})}
{\Phi_{x_0}(\check{\alpha},\check{Z})}\\
\end{eqnarray*}
Now, as mentioned below \eqref{eq:moment of om - Beta-like}
\begin{eqnarray}\label{eq:moment of escaping quenched prob}
\E^{(\alpha+\gamma \theta,Z)} [(\quP_{e_0,\om}[H_{(\partial,x_0)}< H_{e_0}^+])^{-\epsilon}]\le C<\infty
\end{eqnarray}
whenever $\epsilon< \min\{\alpha_{e_0} + \gamma \theta_{e_0}     ,    \alpha_{(\partial,x_0)}+\gamma \theta_{(\partial,x_0)}\}$. In particular, the $-\epsilon$ moment is bounded by $C$
independently of $N$ as long as $\theta_{e_0}\le 1$ and $0<\epsilon < \gamma \theta_{e_0}$.
%
Now consider the Green function $G^N_\om(e_0,e_0)$ of the quenched Markov chain in environment $\om$, killed at the exit time of $B(0,N)$.
We have
$$G^N_\om(e_0,e_0)\le 1/\quP_{e_0,\om}[H_{(\partial,x_0)}< H_{e_0}^+].$$
Indeed, by the Markov property and irreducibility the right hand side equals the Green function at $(e_0,e_0)$ of the walk killed at the hitting time of the edge $(\partial,x_0)$. But the latter can be reached only via exiting $B(0,N)$ and so the inequality holds by coupling.
%
%
Next, we use the Radon-Nikodym derivative \eqref{eq:radon-Nikodym om plus theta} and apply H\"older inequality with $r,q>0$ such that $\frac{1}{q}+\frac{1}{r}=1$:
\sloppy
\begin{footnotesize}
\begin{eqnarray*}
\E^{(\alpha,Z)}\left[ G^N_\om(0,0)^s\right]
&\le&
\E^{(\alpha,Z)}\left[ (\quP_{e_0,\om}[H_{(\partial,x_0)}< H_{e_0}^+])^{-s}\right]\\
&= &
\frac {F(\alpha+\gamma\theta,Z) } {F(\alpha,Z)}
 \E^{(\alpha+\gamma\theta,Z)}\left[ \tilde{u}^{-\gamma\theta}  \left(\quP_{e_0,\om}[H_{(\partial,x_0)}< H_{e_0}^+]\right)^{-s}\right]
\\
&\le &
\frac {F(\alpha+\gamma\theta,Z) } {F(\alpha,Z)}
\E^{(\alpha+\gamma\theta,Z)}\left[ \tilde{u}^{-\gamma q\theta}\right]^{1/q}
\E^{(\alpha+\gamma\theta,Z)}\left[ \left(\quP_{e_0,\om}[H_{(\partial,x_0)}< H_{e_0}^+]\right)^{-rs}\right]^{1/r}
\\
&= &
\frac {F(\alpha+\gamma\theta,Z)^{1-\frac{1}{q}} F(\alpha +(1-q) \gamma \theta ,Z)^{\frac{1}{q}} } {F(\alpha,Z)}
\E^{(\alpha+\gamma\theta,Z)}\left[ \left(\quP_{e_0,\om}[H_{(\partial,x_0)}< H_{e_0}^+]\right)^{-rs}\right]^{1/r}\\
&= &
\frac {F(\alpha+\gamma\theta,Z)^{\frac{1}{r}} F(\alpha +\frac{1}{1-r} \gamma \theta ,Z)^{1-\frac{1}{r}} } {F(\alpha,Z)}
\E^{(\alpha+\gamma\theta,Z)}\left[ \left(\quP_{e_0,\om}[H_{(\partial,x_0)}< H_{e_0}^+]\right)^{-rs}\right]^{1/r}.
\end{eqnarray*}
\end{footnotesize}
To guarantee the right part of the product is bounded by some $C<\infty$ we need to choose $r$ and $\gamma$ so that $rs\le \gamma\theta_{e_0}$, see \eqref{eq:moment of escaping quenched prob}.
 $F(\alpha +\frac{1}{1-r} \gamma \theta ,Z)$ is finite if and only if
\begin{eqnarray}\label{ineq-flow}
\gamma\frac{1}{r-1}\theta(e)<\alpha_e, \;\;\; \forall e\in E_N.
\end{eqnarray}
Since $\theta\le 1$, we can take $\gamma\frac{1}{r-1}<\tilde\kappa$ which means
$r>{\tilde\kappa+\frac{\gamma}{\tilde \kappa}}$.
With such a choice of $r$ we can take
\begin{align}\label{s}
s<\frac{\gamma\tilde\kappa}{\tilde\kappa+\gamma}.
\end{align}
Next $$\frac {F(\alpha+\gamma\theta,Z)^{\frac{1}{r}} F(\alpha +\frac{1}{1-r} \gamma \theta ,Z)^{1-\frac{1}{r}} } {F(\alpha,Z)}=\exp\left(\sum_{x\in V_N}\nu((\alpha^x, \alpha_x),(\gamma\theta_x,\gamma\theta^x),Z_x)\right),$$
where
\begin{small}
\[
\nu((\alpha, \beta),(s, t),Z)=
\frac{1}{r} \log \Phi(\alpha+ s,\beta+ t, Z)
+
(1-\frac{1}{r})\log \Phi(\alpha +\frac{1}{1-r}  s, \beta +\frac{1}{1-r}  t ,Z)
-
\log \Phi(\alpha,\beta,Z),
\]
\end{small}
and $\theta_x=\sum_{e:\overline e =x}$, $\theta^x=\sum_{e:\underline e =x}$, and the corresponding notation for $\alpha$.
(For the dimensions of the domain of $\nu$ the reader would notice that here it is evaluated in
$((\alpha,\beta),(s,t),Z))=(\alpha^x, \alpha_x),(\gamma\theta_x,\gamma\theta^x),Z_x)$.)
In our case,
$(\alpha^x, \alpha_x,Z_x)= (\alpha^0, \alpha_0,Z_0)$
and so that $\{\alpha_e,Z_{e,e'}:0\in{\underline{e},\overline{e}}\}\subset (a,b)$ for some $0<a<b<\infty$.
Note that $\nu((\alpha, \beta) ,(s,t),Z)$ is $C^2$ on every compact subset contained in its domain.
Moreover, we have $\nu(0)$ and $\frac{\mathrm{d}}{\mathrm{d}s_i}\nu=\frac{\mathrm{d}}{\mathrm{d}t_j}\nu=0$
in $(s,t)=(0,0)$.
Therefore, there are $\epsilon,C_r>0$, depending only on $a,b$ such that $|\nu ((\alpha,\beta),(s,t),Z)|\le C_r t^2$ for all $-\epsilon<s,t \le 2d$.
We got that
$$
\E^{(\alpha,Z)}\left[ G^N_\om(0,0)^s\right]\le C\cdot \exp(C_r \|\theta\|^2).
$$
Take a unit flow $\theta$ on $E_N$ from $0$ to $\partial$, such that $0\le\theta\le1$, $\theta_{e_0}>0$, and
$$
\sum_{e\in E_N} \theta_e^2 ={R_N},
$$
where $R_N$ is the electrical resistance between 0 and $B(0,N)^c$ for the network $\Z^d$ with unit resistance on the bonds (see e.g.\ \cite{sabot2016survey}).
In dimension $d\ge 3$, we know that $\sup_N R_N=R(0,\infty)=:\tilde{C}<+\infty$ where
$R(0,\infty)$ is the electrical resistance between $0$ and $\infty$ for unit
resistances on bonds. To sum up, we got
$$
\E^{(\alpha,Z)}\left[ G^N_\om(0,0)^s\right]\le C\cdot \exp(C_r \tilde{C}^2),
$$
for every $s$ satisfying (\ref{s}).
Taking $\gamma$ arbitrarily large we can take $s$ up to $\tilde \kappa$, which completes the proof.

\section{Proof of the invariant measure criterion}\label{subsec:proof of invariant measure criterion}
Proving Theorem \ref{thm:point of view} part \ref{item: kappa smaller than 1} is done by following
\cite{sabot2013particle}[Chapter 5] where in the Proof of Theorem 1(II) there, for transience one uses our
Theorem \ref{thm:green moment}, and in the last paragraph there, instead of the cited Theorem 3 there, one uses our Lemma \ref{lem:leaving edge infty}.

Part \ref{item: kappa larger than 1} of Theorem \ref{thm:point of view} is more involved.
The strategy of the proof is to consider the Radon-Nikodym derivatives $f_N$ of the invariant probability measure
for the process from the point of view of the walker defined on the edges the $N$-torus.
Then, showing that if $p\in[1,\kappa)$ then the $L^p$ norm of $f_N$ with respect to the initial measure on the $N$-torus is uniformly bounded.
This is the content of Lemma \ref{lem:main lemma - bounding f_N} below, where its proof is the main ingredient of the proof.
We shall first state the lemma, following the necessary preparations in Chapters \ref{sec:arc graph identities} and \ref{sec:time reversal}.

For the $N$-torus $T_N$ in $d$ dimensions with arc graph $(E_N,K_N)$ we denote $\Om_N:=\Om_{K_N}$
the corresponding space of environments. It is naturally identified with the space of the $N$-periodic environments on
$\Z^d$. We denote by $\anP^{(\alpha,Z)}_N$ the hypergeometric probability measure on $\Om_N$ defined by \eqref{eq:def of environment law} with parameters $\alpha$ and $Z$. $\E^{(\alpha,Z)}_N$ is
its associated expectation operator. As before, we need to extend the definition to $\anP^{(\alpha,\Theta,Z)}_N$ and $\E^{(\alpha, \Theta,Z)}_N$
whenever $\Theta:K_N\to\reals$ and the measure is well-defined.

For $\om\in\Om_N$ we denote by  $\pi^\om_N=(\pi^\om_N(e))_{e\in E_N}$ the invariant probability measure
of the Markov chain on $E_N$ with transition probabilities $\om$ (it is unique since
the environments are a.s.\ elliptic: $\om(e,e')>0$).

Fix an initial edge $e_0\in\E^d$ so that $\overline{e}_0=0$. For $N\ge 2$, define $f_N:\Om_N\to\reals_+$ by
\begin{equation}
f_N(\om)=2d N^d \cdot \pi^{\om}_N(e_0)
\end{equation}
and
\begin{equation}
\mathbb{Q}^{(\alpha,Z)}_N= f_N \cdot \anP^{(\alpha,Z)}_N.
\end{equation}

\begin{lem}\label{lem:main lemma - bounding f_N}
Let $d\ge 3$. Fix $p\in[1,\kappa)$. Then, $\sup_{N\in\N}\|f_N\|_{ L_p \left( \anP^{(\alpha,Z)}_N \right)}<\infty$.
\end{lem}

Using the Lemma, the proof is standard (see the paragraph after Lemma 1 in Sabot \cite{sabot2013particle}, including the references therein). For convenience we shall give a sketch here.
Consider $\mathbb{Q}^{(\alpha,Z)}_N$ and $\anP^{(\alpha,Z)}_N$ as measures on $N$-periodic environments.
Then, as a product measure over vertices (the matrices $(\om(e,e')_{(\overline{e}=x=\underline{e}')}$, $x\in V$, are i.i.d.)
$\anP^{(\alpha,Z)}_N$ converges weakly to the probability measure $\anP^{(\alpha,Z)}$.
From the definition of $\pi_N^\om(e_0)$ it holds that $\mathbb{Q}^{(\alpha,Z)}_N$ is invariant for the process viewed from the walker on $\Om$.
Since $\Omega$ is compact, then so does the space of product probability measures, and there is an increasing sequence of positive integers and a probability measure so that $\mathbb{Q}^{(\alpha,Z)}_{N_k}\to \mathbb{Q}^{(\alpha,Z)}$.
Since the generator is weakly Feller (i.e.\ continuous with respect to the weak topology), it follows that the weak limit probability measure $\mathbb{Q}^{(\alpha,Z)}$ is invariant for the process viewed from the point of view of the walker on $\Om$.
For every continuous bounded function $g$ on $\Om$, and every $1<p<\kappa$ we have
$$\int g \mathrm{d}\mathbb{Q}^{(\alpha,Z)} \le c_p \|g\|_{L_q(\anP^{(\alpha,Z)})},$$
where $\frac{1}{p}+\frac{1}{q}=1$ (see equation (2.14) in \cite{bolthausen2012ten}).
The last inequality shows that $\mathbb{Q}^{(\alpha,Z)}$ is absolutely continuous with respect to $\anP^{(\alpha,Z)}$, and for $f=\frac{\mathrm{d}\mathbb{Q}^{(\alpha,Z)}}{\mathrm{d}\anP^{(\alpha,Z)}}$ we have $\|f\|_{L_p(\anP^{(\alpha,Z)})} \le c_p$.

\subsection*{Uniformly bounding the Radon-Nikodym derivatives on the torus}\label{sec:pf of Lemma bounding f_N}
In this section we prove Lemma \ref{lem:main lemma - bounding f_N}.
Let $p\in [1,\kappa)$. Combining Lemma \ref{lem:identity for theta to the div over the same with the checks} and Lemma \ref{lem:identity for pi to the div if theta is a total flow},
if $\Theta_N:K_N\to\reals_+$ is satisfying
\begin{equation}\label{eq:condition of total flow of strength p over 2d N tothe d}
\mathrm{div}(\Theta)=\frac{p}{dN^d}\displaystyle\sum_{\tilde{e}\in E_N}(\delta_{e_0}-\delta_{\tilde{e}})
\end{equation}
then
\begin{equation}\label{eq:f N tothe p is notmore than}
f_N^p(\om)\le\frac{\check{\omega}^{\check{\Theta}_N}}{\omega^{\Theta_N}}.
\end{equation}
Remember that the root edge $e_0$ was chosen such that $\overline e_0=0$.
In the sequel we will often simply write $e_i$ for the directed edge $(0,e_i)$
(remember that $e_1, \ldots, e_d$ is the base of $\reals^d$).
Now, by H\"{o}lder inequality, $1\le(2d)^\kappa\displaystyle\sum_{i=1}^{2d}\om_i^\kappa$ where $\om_i:=\om(e_0, e_i)$. 
Therefore,
$$\E^{(\alpha,Z)}[f_N^p]\le (2d)^\kappa \displaystyle\sum_{i=1}^{2d} \E^{(\alpha,Z)}[\om_i^\kappa f_N^p].$$
Hence, from \eqref{eq:f N tothe p is notmore than}, Lemma \ref{lem:main lemma - bounding f_N} follows once we show that
for every $1\le i\le 2d$ and $N\in \N$ there is $\Theta_N:K_N\to\reals_+$ satisfying \eqref{eq:condition of total flow of strength p over 2d N tothe d}, so that
\begin{equation}\label{eq:main step in main lemma}
\sup_{N\in\N} \E^{(\alpha,Z)}\left[\check{\omega}^{\check{\Theta}_N} \om_i^\kappa \omega^{-\Theta_N}\right] < \infty.
\end{equation}
We shall now prove \eqref{eq:main step in main lemma}. Let $\alpha^{(i)}$, $1\le i\le 2d$, be the weights defined by $\alpha^{(i)}:=\alpha + \kappa \one_{e=e_i}$. I.e.\ $\alpha^{(i)}$ gives $\alpha$ an extra $\kappa$ on the specific edge $e_i$ but leaves it unchanged on all other edges. Then,
\begin{equation}\label{eq:m of alpha i is larger than kappa}
m(\alpha^{(i)})\ge \kappa
\end{equation}
where $m(c)$ is the min cut of the network $(c(e))_{e\in\E(\Z^d)}$ on $\Z^d$ (that is, the minimal $c$-weight of a set separating $0$ from $\infty$), see equation (3.10) and the paragraph below it in \cite{sabot2013particle} for the proof.
We shall now show \eqref{eq:main step in main lemma} for the case $i=1$. The other $2d-1$ possibilities are symmetric. Fix $N\ge 1$, and apply Lemma \ref{lem: Construction of total flow of edges}
with $c(e)=\alpha^{(1)}(e)$ to get $\tilde{\Theta}=\tilde{\Theta}_N$ with bounded $L_2$ norm,
almost below the capacity
\begin{equation}\label{eq:total outer of theta tilde}
\overline{\tilde{\Theta}}_e \le \alpha^{(1)}(e)+ m(\alpha^{(1)}) \one_{e=e_0}
\end{equation}
so that it is a total flow from $e_0$ with strength $\frac{m(\alpha^{(1)})}{dN^{d}}$.
Set
\begin{equation}\label{eq: def of Theta from tildeTheta}
\Theta =\Theta_N := \frac{p}{m(\alpha^{(1)})}\tilde{\Theta}.
\end{equation}
Then $\Theta$ is also total flow from $e_0$ with a bounded $L_2$ norm and with
strength $\frac{p}{dN^{d}}$.

Remember the notation $\beta^{\gamma}$ from \eqref{eq:def of power of functions}. Fix $q=q(\alpha,d)>0$ to be chosen later-on. Let $r>0$ be so that
$\frac{1}{r} + \frac{1}{q} = 1$.
Using H\"{o}lder inequality, the Weak Reversibility Corollary \ref{cor:time reversal environment},
we have
\begin{eqnarray*}
\E^{(\alpha,Z)}[\check{\om}^{\check{\Theta}} \om_1^\kappa \om^{-\Theta}] &\le &
\E^{(\alpha,Z)}[\check{\om}^{r\check{\Theta}}]^{1/r} \E^{(\alpha,Z)}[\om_1^{q\kappa} \om^{-q\Theta}]^{1/q} \\
&=&
\E^{(\check{\alpha},\check{Z})}[\om^{r\check{\Theta}}]^{1/r} \E^{(\alpha,Z)}[\om_1^{q\kappa}\om^{-q\Theta}]^{1/q}.
\end{eqnarray*}

Assume for the moment that the functions $F$ and $G$ below are well defined.
This will be justified by a suitable choice of $q$.
Using \eqref{eq:expectation of om to the xi} we get:
\begin{eqnarray*}
\E^{(\check{\alpha},\check{Z})}[\om^{r\check{\Theta}}]^{1/r} \E^{(\alpha,Z)}[\om_1^{q\kappa}\om^{-q\Theta}]^{1/q} &=&
\check{Z}^{\check{\Theta}}  \cdot Z^{-\Theta} \times\\
&\times& \left(\frac {F(\check{\alpha}; r \check{\Theta};\check{Z})} {F(\check{\alpha}; \check{Z})}   \right)^{1/r}
\times
\left(\frac {F( \alpha + q\kappa\delta_{e_1}; -q\Theta ;Z) } {F(\alpha;Z)}   \right)^{1/q}\\
  &=& \left(\frac {F(\check{\alpha}; r \check{\Theta};\check{Z})} {F(\check{\alpha}; \check{Z})}   \right)^{1/r}
\times
\left(\frac {F( \alpha + q\kappa\delta_{e_1}; -q\Theta ;Z) } {F(\alpha;Z)}   \right)^{1/q}.
\end{eqnarray*}
Using the Duality Lemma \ref{lem:duality} for the term with power $1/r$, together with the fact that $\check{\alpha}_{\check{e}}=\alpha_e$ and $\overline{\check{\Theta}}_{\check{e}'}= \underline{\Theta}_{e'}$, the last product equals
\begin{eqnarray*}
\left(\frac{G(\alpha+r\underline{\Theta})}
{G(\alpha+r\overline{\Theta})}\right)^{1/r}
&\times&
\frac {F( \alpha ; r \Theta ; Z)^{1/r} F( \alpha + q\kappa\delta_{e_1}; -q \Theta ; Z)^{1/q}}{F(\alpha ; Z)}.
\end{eqnarray*}

\emph{Choice of $q$}: The terms in the products above will be well-defined if all the terms evaluated by $F$ and $G$ are strictly positive.
Let us see what should $q>0$ satisfy to achieve that. First note that the terms with power $1/r$
are strictly positive since so is $\alpha$, whereas $\Theta$ is non-negative.
For the terms with power $1/q$ to be strictly positive, we need to have
\begin{equation}\label{eq:ineq1 with q,e}
\alpha_e - q\overline{\Theta}_e + q\kappa\one_{e=e_1} > 0
\end{equation} and
\begin{equation}\label{eq:ineq2 with q,e'}
\alpha_{e} - q\underline{\Theta}_{e} + q\kappa\one_{e=e_1} > 0.
\end{equation}
%
Equation \eqref{eq: def of Theta from tildeTheta} gives $\overline{\Theta}_e =\frac{p}{m(\alpha^{(1)})}\overline{\tilde{\Theta}}_e$.
From \eqref{eq:m of alpha i is larger than kappa} $p<\kappa\le m(\alpha^{(1)})$, and using
\eqref{eq:total outer of theta tilde}, and the definition of $\alpha^{(1)}$, we get
\begin{eqnarray*}
\overline{\Theta}_e
\le \frac{p}{\kappa} \alpha_e + p\one_{e=e_1\text{ or }e=e_0} \le \alpha_e + \kappa \text{ and }
\overline{\Theta}_x
\le \frac{p}{\kappa} \overline{\alpha}_0 + p \one_{x=0\text{ or }x=\overline{e}_1}\le \overline{\alpha}_0 + \kappa.
\end{eqnarray*}
Let $a:=\min\{\alpha_{e_i}: 1\le i \le 2d\}>0$, $b:=\max\{\alpha_{e_i}:1\le i \le 2d\}$, and $B:= \max\{b, \kappa\}<\infty$. Then
\begin{eqnarray*}
\overline{\Theta}_e \le 2B \text{  and  }
\overline{\Theta}_x \le db + \kappa \le (d+1) B.
\end{eqnarray*}
$\Theta$ is a total flow from $e_0$ of strength $\frac{p}{dN^d}$, and therefore
$$\underline{\Theta}_e = \overline{\Theta}_e - p\one_{e=e_0} + \frac{p}{dN^d}\one_{e\ne e_0}
\le 2B + \frac{p}{dN^d}\le 3B.$$
Noting that $a\le \alpha_e < \overline{\alpha}_0$, and choosing $q=q(\alpha,d)>0$ to satisfy $$q<\frac{a}{(d+1)B},$$
then \eqref{eq:ineq1 with q,e} and \eqref{eq:ineq2 with q,e'} 
follow, and so we have shown well-definability.

Next, since $a\le\alpha_{e}\le b$, and $\log\Gamma$ is $C^1$ on $\reals_*$ (e.g., since the digamma function is holomorphic on $\mathbb{C}\backslash \{0,-1,-2,...\}$),
we have that it is Lipschitz in the domain, i.e.\ there is some constant $c_3=c_3(\alpha,d)$ so that
\begin{equation}\label{eq:LogGamma estimate}
\frac{\Gamma(s + t + h))}{\Gamma(s + t)}=e^{\log{\Gamma(s+ t + h))}-\log{\Gamma(s + t)}}\le e^{c_3\cdot |h|}
\end{equation}
for all $s\in [a,b]$ and $t,t+h\in [0,b]$.
Now, by \eqref{eq:condition of total flow of strength p over 2d N tothe d} $\dive(\Theta)= \overline\Theta -\underline\Theta$ is proportional to the volume of the box,
and so by \eqref{eq:LogGamma estimate} we have
$$
\prod_{e\in E_N} \Gamma(\alpha_e + r \underline{\Theta}_e) /  \Gamma(\alpha_e + r \overline{\Theta}_e))\le
\exp\left(\frac{c_2 r p(dN^d-1)}{dN^d}\right)  \exp\left( c_2 r \frac{p}{dN^d} \right)^{\#\{{e\in E_N}, e\neq e_0\}}
\le e^{ 2 c_3 r p }.
$$
Similarly, since $r\sum_{\underline{e} = x} (\overline{\Theta}_e - \underline{\Theta}_e)= 2d r
\cdot \sum_{\underline{e} = x} \dive(\Theta)(e)$, by dividing to the to cases $x=\underline(e)_0$ and $x\ne \underline(e)_0$ and
using \eqref{eq:condition of total flow of strength p over 2d N tothe d},
we have that
$$
\prod_{x\in V } \Gamma\left(\sum_{\underline{e} = x} (\alpha_e + r \underline{\Theta}_e)\right) /
\Gamma\left(\sum_{\underline{e} = x} (\alpha_e + r \overline{\Theta}_e)\right)
\le e^{ 4d^2 c_3 r p }.
$$
To sum up, so far we have
$\left(\frac{G(\alpha+r\underline{\Theta})}
{G(\alpha+r\overline{\Theta})}\right)^{1/r}\le
(e^{4d^2 c_3 r p}) ^{1/r}=e^{4d^2 c_3 p}=:e^{c_4}$.
Therefore
\begin{eqnarray*}
\E^{(\alpha,Z)}[\om_1^\kappa \check{\omega}^{\check{\Theta}} \omega^{-\Theta}] \le
\exp({c_4})\exp\left(\displaystyle\sum_{x\in T_N} \nu (\alpha^x, \alpha_x,  \Theta_x)  \right),
\end{eqnarray*}
where $\Theta_x:=\Theta(e,e'))_{\overline{e}=x=\underline{e}'}$, and
\begin{align*}
\nu (\alpha^x, \alpha_x,  \Theta_x) :=&
 \frac{1}{r} \log \Phi( ({\alpha}_{{e}} + r\overline{{\Theta}}_{{e}})_{\overline{{e}}=x} ,({\alpha}_{{e}'}+ r\underline{{\Theta}}_{{e}'})_{\underline{{e}}'=x} ,Z)  \\
 +&  \frac{1}{q} \log \Phi( (\alpha_e - q\overline{\Theta}_e + q\kappa\one_{e=e_1})_{\overline{e}=x} ,(\alpha_{e'} - q\underline{\Theta}_{e'} + q\kappa\one_{e'=e_1})_{\underline{e}'=x} ,Z) \\
 -& \log \Phi( (\alpha_e)_{\overline{e}=x} ,(\alpha_{e'})_{\underline{e}'=x} ,Z).
\end{align*}
Note that $\nu : [a,b] ^{2d} \times [a,b]^{2d} \times [0,C]^2 \to \reals_+$ is $C^2$ on a compact set.
Moreover, it satisfies $\nu ( \alpha^x, \alpha_x,  \Theta_x ) = 0$ at $\Theta_x = 0$ and
$\frac{\partial}{{\partial\Theta(e,e')}}\nu (\alpha^x, \alpha_x,  \Theta_x = 0)=0$ at $\Theta_x = 0$.
By a $2$nd order $2d$-dimensional Taylor Theorem,
there is a constant $c_5=c_5(\alpha,d)$ so that we have
$$\nu (\alpha^x, \alpha_x,  \Theta_x )\le c_5 \displaystyle\sum_{\overline{e}=x=\underline{e}'}\Theta(e,e')^2.$$
But by construction $\Theta$ has a bounded $L_2$ norm with some constant $c_1$. Therefore,
\begin{eqnarray*}
\E^{(\alpha,Z)}[\om_1^\kappa \check{\omega}^{\check{\Theta}} \omega^{-\Theta}] \le
e^{c_4}\exp\left(c_5 \displaystyle\sum_{e,e'\in K_N}\Theta(e,e')^2 )   \right) \le \exp(c_4+c_5 c_1).
\end{eqnarray*}
This concludes the proof of the lemma.


\begin{appendices}

\section{Duality of hypergeometric functions}\label{appendix:Duality of hypergeometric functions}

In this section we give a direct proof for Lemma \ref{lem:duality} on the duality relation for hypergeometric functions. Note that for every $t,\beta>0$
\begin{equation}\label{eq:gamma function identity}
\frac{1}{\Gamma(\beta)}\int_0^\infty e^{-tv}v^{\beta-1}dv = t^{-\beta}.
\end{equation}
Recall \eqref{eq:hypergeometric density} and \eqref{eq:hypergeometric function}.
The strategy is to first use \eqref{eq:gamma function identity} to construct a variable $v$
that will take a dual role of $u$ and then to add another variable to ``free the variable $u$ from the simplex''.
The next step is to modify $u$ and $v$ to make the integral suitable for duality.
The conclusion is by following the above steps in a reverse order with the new $v$ and $u$.
Here is the calculation in detail followed by some clarifications.
\begin{small}
\begin{align*}
\Phi(\alpha,\beta,Z) =& \int_{\Delta^{(n)}} \varphi (\alpha,\beta;Z;u) du\\
=& \int_{\Delta^{(n)}} \left( \displaystyle\prod_{i=1}^n u_i^{\alpha_i - 1} \right) \displaystyle\prod_{j=1}^l  \left( (Z\cdot u)_j ^{-\beta_j} \right) du\\
=& \int_{\Delta^{(n)}} \left( \prod_i u_i^{\alpha_i - 1} \right) \int_{\reals_+^l} \left( \prod_j \frac{v_j^{\beta_j-1}}{\Gamma(\beta_j)}\right) e^{-<v,Z\cdot u>} dv du \\
=& \frac{1}{\prod_j \Gamma(\beta_j)}\int_{\Delta^{(n)}} \int_{\reals_+^l} \left(\prod_j v_j^{\beta_j-1}\right)\left(\prod_i u_i^{\alpha_i-1}\right)e^{-<v,Z\cdot u>}dudv\\
=& \frac{1}{\prod_j\Gamma(\beta_j)} \frac{1}{ \Gamma\left( \sum_i \alpha_i \right)}\int_{\reals_+^{n}} \int_{\reals_+^l} e^{-\sum_i u_i}\left(\prod_j v_j^{\beta_j-1}\right)\left(\prod_i u_i^{\alpha_i-1}\right)e^{-<\frac{1}{\sum_i u_i}v,Z\cdot u>}dudv\\
=& \frac{1}{\prod_j\Gamma(\beta_j)} \frac{1}{ \Gamma\left( \sum_i \alpha_i \right)}\int_{\reals_+^{n}} \int_{\reals_+^l} e^{-\sum_j v_j}\left(\prod_j v_j^{\beta_j-1}\right)\left(\prod_i u_i^{\alpha_i-1}\right)e^{-<\frac{1}{\sum_j v_j} u,Z^t\cdot v>}dvdu\\
=& \frac{\prod_i\Gamma(\alpha_i)}{\prod_j \Gamma(\beta_j)} \frac{\Gamma\left( \sum_j \beta_j \right)}{ \Gamma\left( \sum_i \alpha_i \right)}
\int_{\Delta^{(l)}} \left( \prod_j u_j^{\beta_j - 1} \right) \prod_i (Z^t\cdot u)_j ^{-\beta_j}du\\
=& \frac{\prod_i \Gamma(\alpha_i)}{\prod_j \Gamma(\beta_j)} \frac{\Gamma\left( \sum_j \beta_j \right)}{ \Gamma\left( \sum_i \alpha_i \right)}
\Phi(\beta,\alpha,Z^t)
\end{align*}
\end{small}
The third equality follows from \eqref{eq:gamma function identity}.
For the fifth equality, note that using the change of variables $\lambda=\sum_i w_i,\, u_i=\frac{1}{\lambda}w_i$, we have that
%
\begin{align*}
\int_{w\in\reals_+^{n}} \int_{v\in\reals_+^l} e^{-\sum_i w_i}\left(\prod_j v_j^{\beta_j-1}\right)\left(\prod_i w_i^{\alpha_i-1}\right)e^{-<\frac{1}{\sum_i w_i}v,Z\cdot w>}dwdv \\
 = \int_{\lambda\in \reals_+} e^{-\lambda} \cdot \lambda^{(\sum_i \alpha_i -1)} \int_{u\in\Delta_n} \int_{\reals_+^l} e^{-\sum_i u_i}\left(\prod_j v_j^{\beta_j-1}\right) \left(\prod_i u_i^{\alpha_i-1}\right)e^{-<v,Z\cdot u>}d\lambda dudv
\\
=  \Gamma\left(\sum_i \alpha_i\right)\int_{u\in\Delta^{(n)}} \int_{v\in\reals_+^l} \left(\prod_j v_j^{\beta_j-1}\right)\left(\prod_i u_i^{\alpha_i-1}\right)e^{-<v,Z\cdot u>}dudv.
\end{align*}
To see the sixth equality, make a change of variables $u \to \tilde{u} = \frac{\sum_i v_j}{\sum_i u_i} u$ and $v \to \tilde{v} = \frac{\sum_i u_i}{\sum_j v_j} v$, and deduce the equality from the fact that $\sum_i \alpha_i = \sum_j \beta_j$. The one before last equality follows from the previous equalities by interchanging the roles of $(\alpha,u,n,Z)$ and $(\beta,v,l,Z^t)$. The last equality follows from the definition of $\Phi$. This gives the desired duality.

%
%


 \end{appendices}

\section*{Acknowledgements}{This work was performed within the framework of the LABEX MILYON (ANR-10- LABX-0070) of Universit\'e de Lyon, within the program ``Investissements d'Avenir'' (ANR-11-IDEX-0007) operated by the French National Research Agency (ANR).
T.O.\ is also gratefully acknowledge financial
support from the German Research Foundation through the research unit FOR 2402 – Rough paths, stochastic partial differential equations and related topics - at Humboldt University, Technische Universit\"at Berlin and Weierstrass Institute.}

 \small

\bibliography{Dirichlet}{}

\begin{thebibliography}{AKKI11}

\bibitem[AKKI11]{aomoto2011theory}
K.~Aomoto, M.~Kita, T.~Kohno, and K.~Iohara.
\newblock {\em Theory of hypergeometric functions}.
\newblock Springer, 2011.

\bibitem[BCR16]{Berger-Rosenthal-16}
Noam Berger, Moran Cohen, and Ron Rosenthal.
\newblock Local limit theorem and equivalence of dynamic and static points of
  view for certain ballistic random walks in i.i.d. environments.
\newblock {\em Ann. Probab.}, 44(4):2889--2979, 2016.

\bibitem[BD14]{berger2014quenched}
Noam Berger and Jean-Dominique Deuschel.
\newblock A quenched invariance principle for non-elliptic random walk in iid
  balanced random environment.
\newblock {\em Probab. Theory Related Fields}, 158(1-2):91--126, 2014.

\bibitem[BDR14]{Berger-Ramirez-14}
Noam Berger, Alexander Drewitz, and Alejandro~F. Ramirez.
\newblock Effective polynomial ballisticity conditions for random walk in
  random environment.
\newblock {\em Comm. Pure Appl. Math.}, 67(12):1947--1973, 2014.

\bibitem[Bou13]{Bouchet-13}
\'Elodie Bouchet.
\newblock Sub-ballistic random walk in {D}irichlet environment.
\newblock {\em Electron. J. Probab.}, 18:no. 58, 25, 2013.

\bibitem[BS12]{bolthausen2012ten}
Erwin Bolthausen and Alain-Sol Sznitman.
\newblock {\em Ten lectures on random media}, volume~32.
\newblock Birkh{\"a}user, 2012.

\bibitem[BZ07]{Bolthausen-Zeitouni-07}
Erwin Bolthausen and Ofer Zeitouni.
\newblock Multiscale analysis of exit distributions for random walks in random
  environments.
\newblock {\em Probab. Theory Related Fields}, 138(3-4):581--645, 2007.

\bibitem[BZ08]{Berger-Zeitouni-08}
Noam Berger and Ofer Zeitouni.
\newblock A quenched invariance principle for certain ballistic random walks in
  i.i.d. environments.
\newblock In {\em In and out of equilibrium. 2}, volume~60 of {\em Progr.
  Probab.}, pages 137--160. Birkh\"auser, Basel, 2008.

\bibitem[CL91]{chamayou_letac_1991}
Jean-Fran\c{c}ois Chamayou and G\'erard Letac.
\newblock Explicit stationary distributions for compositions of random
  functions and products of random matrices.
\newblock {\em J. Theoret. Probab.}, 4(1):3--36, 1991.

\bibitem[ES06]{Enriquez-Sabot-06}
Nathana\"el Enriquez and Christophe Sabot.
\newblock Random walks in a {D}irichlet environment.
\newblock {\em Electron. J. Probab.}, 11:no. 31, 802--817, 2006.

\bibitem[GZ12]{guo2012quenched}
Xiaoqin Guo and Ofer Zeitouni.
\newblock Quenched invariance principle for random walks in balanced random
  environment.
\newblock {\em Probab. Theory Related Fields}, 152(1-2):207--230, 2012.

\bibitem[Koz85]{kozlov1985method}
S.~M. Kozlov.
\newblock The method of averaging and walks in inhomogeneous environments.
\newblock {\em Russian Mathematical Surveys}, 40(2):73--145, 1985.

\bibitem[KV86]{kipnis1986central}
C.~Kipnis and S.~R.~S. Varadhan.
\newblock Central limit theorem for additive functionals of reversible markov
  processes and applications to simple exclusions.
\newblock {\em Commun. Math. Phys.}, 104(1):1--19, 1986.

\bibitem[Law82]{lawler1982weak}
Gregory~F Lawler.
\newblock Weak convergence of a random walk in a random environment.
\newblock {\em Commun. Math. Phys.}, 87(1):81--87, 1982.

\bibitem[RA03]{Rassoul-Agha-03}
Firas Rassoul-Agha.
\newblock The point of view of the particle on the law of large numbers for
  random walks in a mixing random environment.
\newblock {\em Ann. Probab.}, 31(3):1441--1463, 2003.

\bibitem[RAS09]{Rassoul-Seppalainen-09}
Firas Rassoul-Agha and Timo Sepp\"al\"ainen.
\newblock Almost sure functional central limit theorem for ballistic random
  walk in random environment.
\newblock {\em Ann. Inst. Henri Poincar\'e Probab. Stat.}, 45(2):373--420,
  2009.

\bibitem[Sab11]{sabot2011transience}
Christophe Sabot.
\newblock Random walks in random dirichlet environment are transient in
  dimension $d\ge3$.
\newblock {\em Probab. Theory Related Fields}, 151(1):297--317, 2011.

\bibitem[Sab13]{sabot2013particle}
Christophe Sabot.
\newblock Random dirichlet environment viewed from the particle in dimension
  $d\ge 3$.
\newblock {\em Ann. Probab.}, 41(2):722--743, 2013.

\bibitem[ST11]{Sabot-Tournier-11}
Christophe Sabot and Laurent Tournier.
\newblock Reversed {D}irichlet environment and directional transience of random
  walks in {D}irichlet environment.
\newblock {\em Ann. Inst. Henri Poincar\'e Probab. Stat.}, 47(1):1--8, 2011.

\bibitem[ST17]{sabot2016survey}
Christophe Sabot and Laurent Tournier.
\newblock Random walks in dirichlet environment: an overview.
\newblock {\em Ann. Fac. Sci. Toulouse Math. (6)}, 26(2):463--509, 2017.

\bibitem[SZ99]{Sznitman-Zerner-99}
Alain-Sol Sznitman and Martin Zerner.
\newblock A law of large numbers for random walks in random environment.
\newblock {\em Ann. Probab.}, 27(4):1851--1869, 1999.

\bibitem[SZ06]{Sznitman-Zeitouni-2006}
Alain-Sol Sznitman and Ofer Zeitouni.
\newblock An invariance principle for isotropic diffusions in random
  environment.
\newblock {\em Invent. Math.}, 164(3):455--567, 2006.

\bibitem[Szn00]{Sznitman-JEMS-00}
Alain-Sol Sznitman.
\newblock Slowdown estimates and central limit theorem for random walks in
  random environment.
\newblock {\em J. Eur. Math. Soc. (JEMS)}, 2(2):93--143, 2000.

\bibitem[Szn02]{Sznitman-PTRF-02}
Alain-Sol Sznitman.
\newblock An effective criterion for ballistic behavior of random walks in
  random environment.
\newblock {\em Probab. Theory Related Fields}, 122(4):509--544, 2002.

\bibitem[Tou09]{tournier2009integrability}
Laurent Tournier.
\newblock Integrability of exit times and ballisticity for random walks in
  dirichlet environment.
\newblock {\em Electron. J. Probab}, 14(16):431--451, 2009.

\end{thebibliography}
\bibliographystyle{alpha}

\end{document}